\documentclass[12pt,reqno]{amsart}
\usepackage[margin=1in]{geometry}
\makeatletter
\def\section{\@startsection{section}{1}%
	\z@{.7\linespacing\@plus\linespacing}{.5\linespacing}%
	{\bfseries
		\centering
}}
\def\@secnumfont{\bfseries}
\makeatother
\usepackage{amsmath,amssymb,amsthm,graphicx,amsxtra, setspace}
\usepackage[utf8]{inputenc}
\usepackage{mathrsfs}
\usepackage{alltt}
\usepackage{relsize}
\usepackage{hyperref}
\usepackage{aliascnt}
\usepackage{tikz}
\usepackage{mathtools}
\usepackage{multicol}
\usepackage{upgreek}
\usepackage{graphicx,type1cm,eso-pic,color}
\allowdisplaybreaks

\newtheorem{theorem}{Theorem}[section]
\newtheorem{lemma}[theorem]{Lemma}

\newtheorem{definition}[theorem]{Definition}

\newtheorem{remark}[theorem]{Remark}
\newtheorem{hypothesis}[theorem]{Hypothesis}

\renewcommand{\d}{\/\mathrm{d}\/}

\def\w{\textbf{W}^{\varepsilon}_{{\theta}^{\varepsilon}}}

\def\L{\mathbb{L}}
\def\wi{\widetilde}
\def\X{\mathbb{X}}
\def\A{\mathrm{A}}
\def\G{\mathrm{G}}

\def\g{\mathbf{g}}

\def\C{\mathrm{C}}
\def\h{\mathbf{h}}
\def\B{\mathrm{B}}

\def\D{\mathrm{D}}

\def\y{\mathbf{y}}

\def\x{\mathbf{x}}
\def\z{\mathbf{z}}
\def\p{\mathbf{p}}

\def\v{\mathbf{v}}
\def\w{\mathbf{w}}
\def\W{\mathrm{W}}

\def\f{\mathbf{f}}
\def\V{\mathbb{V}}
\def\U{\mathrm{U}}

\def\u{\mathbf{u}}

\def\H{\mathbb{H}}
\def\n{\mathbf{n}}

\newcommand{\R}{\mathbb{R}}

\renewcommand{\d}{\/\mathrm{d}\/}

\let\originalleft\left
\let\originalright\right
\renewcommand{\left}{\mathopen{}\mathclose\bgroup\originalleft}
\renewcommand{\right}{\aftergroup\egroup\originalright}

\newcommand{\Addresses}{{
		\footnote{
			
			\noindent \textsuperscript{1}Department of Mathematics, Indian Institute of Technology Roorkee-IIT Roorkee,
			Haridwar Highway, Roorkee, Uttarakhand 247667, INDIA.\par\nopagebreak
			\noindent  \textit{e-mail:} \texttt{maniltmohan@ma.iitr.ac.in, maniltmohan@gmail.com.}
			
			\noindent \textsuperscript{*}Corresponding author.

			\textit{Key words:} porus medium, convective Brinkman-Forchheimer equations, first order necessary conditions, Pontryagin's maximum principle, optimal control.
			
			Mathematics Subject Classification (2020): 49J20, 35Q35, 76D03.

}}}

\begin{document}

	\title[Optimal Control of the 2D CBF equations]{Optimal control problems governed by two dimensional convective Brinkman-Forchheimer equations\Addresses}
	\author[M. T. Mohan]
	{Manil T. Mohan\textsuperscript{1*}}

	\maketitle
	
	\begin{abstract}
	The  convective Brinkman-Forchheimer (CBF) equations describe the motion of incompressible viscous fluid  through a rigid, homogeneous, isotropic, porous medium and is given by
	$$\partial_t\u-\mu \Delta\u+(\u\cdot\nabla)\u+\alpha\u+\beta|\u|^{r-1}\u+\nabla p=\f,\ \nabla\cdot\u=0.$$
	In this work, we consider some distributed optimal control problems like total energy minimization, minimization of enstrophy, etc governed by the two dimensional CBF equations with the absorption exponent $r=1,2$ and $3$. We show  the existence  of an optimal solution and  the first order necessary conditions of optimality for such   optimal control problems in terms of the Euler-Lagrange system. Furthermore, for the case $r=3$,  we show the second order necessary and sufficient conditions of optimality.  We also investigate an another control problem which is similar to that of the data assimilation problems in meteorology of obtaining unknown initial data, when the system under consideration is 2D CBF equations, using optimal control techniques.  
	\end{abstract}

	\section{Introduction}\label{sec1}\setcounter{equation}{0}  
The  convective Brinkman-Forchheimer (CBF) equations  describe the motion of  incompressible viscous fluid  through a rigid, homogeneous, isotropic, porous medium.	In this work, we consider the controlled CBF  equations and investigate some optimal control problems. Let $\mathcal{O}\subset\R^2$ be a bounded domain  with a smooth boundary $\partial\mathcal{O}$. Let $\u(t , x) \in \R^2$ denote the velocity field at time $t$ and position $x$,  $p(t,x)\in\R$ represent the pressure field, and $\f(\cdot,\cdot)\in\R^2$ stand for an external forcing. We consider the following controlled CBF equations (see \cite{AKS} for physical interpretation of the model): 
\begin{equation}
\begin{aligned}\label{1}
\frac{\partial \u(t,x)}{\partial t}&-\mu\Delta\u(t,x)+(\u(t,x)\cdot\nabla)\u(t,x)+\alpha\u(t,x)+\beta|\u(t,x)|^{r-1}\u(t,x)+\nabla p(t,x)\\&\quad=\f(t,x)+\D\U(t,x),\ \text{ for } x\in\mathcal{O},\ t>0,
\end{aligned} 
\end{equation}
with the conditions
\begin{equation}\label{1.2}
\left\{
\begin{array}{ll}
(\nabla\cdot\u)(t,x)=0, &  \text{ for }\ x\in\mathcal{O},\ t>0,\\
\u(t,x)=\mathbf{0}, &  \text{ for }\ x\in\partial\mathcal{O},\ t\geq 0,\\
\u(0,x)=\u_0(x),& \text{ for }\ x\in\mathcal{O},\\
\int_{\mathcal{O}}p(t,x)\d x=0,&\text{ for } \ t\geq 0.
\end{array}
\right.
\end{equation}
The system \eqref{1} can be considered as a modification (by an absorption term $\alpha\u+\beta|\u|^{r-1}\u$) of the classical controlled Navier-Stokes equations (NSE). That is,  for $\alpha=\beta=0$, we obtain the classical controlled 2D NSE.  The final condition in \eqref{1.2} is imposed for the uniqueness of the pressure $p$. The constant $\mu$ represents the positive Brinkman coefficient (effective viscosity), the positive constants $\alpha$ and $\beta$ represent the Darcy (permeability of porous medium) and Forchheimer (proportional to the porosity of the material) coefficient, respectively and the absorption exponent $r\in[1,\infty)$.  The \emph{distributed control} acting on the system is denoted by $\U(\cdot,\cdot)$ in \eqref{1}  and $\D$ is a bounded linear operator.  The model \eqref{1.2} is recognized to be more accurate when the flow velocity is too large for the Darcy's law to be valid alone, and in addition, the porosity is not too small. 

Let us first discuss about the global solvability results  for the uncontrolled  system \eqref{1} (with $\U=\mathbf{0}$)  available  in the literature. The existence of a global weak solution to the system \eqref{1} in two and three dimensional bounded domains is obtained in \cite{SNA}. The authors have also established the uniqueness of weak solutions in two dimensions for $r\in[1,\infty)$. The Brinkman-Forchheimer equations with fast growing nonlinearities is investigated in \cite{KT2} and the authors established the existence of regular dissipative solutions and global attractors for the system \eqref{1} in three dimensions for $r> 3$. This result assures the existence of global weak solutions to the 3D CBF equations  in the Leray-Hopf sense satisfying the energy equality. The authors in \cite{CLF} showed that all weak solutions of the critical CBF equations ($r=3$ with $4\beta\mu\geq 1$) posed on a bounded domain in $\mathbb{R}^3$ satisfy the energy equality. The existence and uniqueness of a global weak solution in the Leray-Hopf sense satisfying the energy equality to the  3D CBF equations with $r\ge 3$ ($2\beta\mu\geq 1$, for $r=3$)  is investigated in \cite{MTM7} by exploiting  the monotonicity as well as the demicontinuity properties of the linear and nonlinear operators, and the Minty-Browder technique.  

 Control of  fluid flow and turbulence inside a flow in a given physical domain, with known initial data and by means of body forces, boundary data, temperature are well-studied problems in fluid mechanics  (cf. \cite{FART,fursikov,gunzburger,sritharan} etc).  The optimal control problems governed by the Navier-Stokes equations have been developed extensively for the past few decades after the mathematical advancements in the optimal control of infinite dimensional nonlinear system theory and partial differential equations (see for example \cite{FART,VB3,SDMTS,fursikov,gunzburger,lions,raymond,sritharan} etc). Various optimal control problems governed by the Navier-Stokes equations have been considered in \cite{VB2,HOF,GWa,GWa1}, etc and the references therein.  In this work, we consider the optimal control problems like total energy minimization, minimization of enstrophy, etc, governed by the two dimensional CBF equations \eqref{1}. We show  the existence  of an optimal solution as well as  the first order necessary conditions of optimality (Pontryagin's maximum principle) for such   optimal control problems in terms of the Euler-Lagrange system. For the cubic growth ($r=3$), we show the second order necessary and sufficient conditions of optimality for such problems. Using optimal control techniques, we  investigate an another control problem which is similar to that of data assimilation problems in meteorology of obtaining unknown initial data, when the system under consideration is 2D CBF equations \eqref{1}.  Due to technical difficulties (see Remark \ref{rem4.4} below), in this work, we take the absorption exponent $r=1,2$ and $3$ only (physically relevant cases). The case of $r=1$ is not discussed in details as it is just a perturbation of the 2D NSE with $(\alpha+\beta)\u$ and not much technical difficulties arise. 
 
 The rest of the paper is organized as follows. In the next section, we provide the necessary function spaces needed to obtain the existence of a unique weak solution satisfying the energy equality for  the system \eqref{1}. In the same section, we also discuss about some important properties of the  linear and nonlinear operators like local monotonicity, hemicontinuity, etc, which is used to obtain the global solvability results (Theorems \ref{thm2.2}, \ref{thm2.3} \ref{main2}, etc). We consider the linearized system corresponding to \eqref{1} and establish the global solvability results in the same section (Theorem \ref{linear}). A distributed optimal control problem  of minimizing the total energy  as well as enstrophy subjected to the controlled 2D CBF equations \eqref{1} is formulated in section \ref{sec3}. We show  the existence  of an optimal solution for such problems in the same section (Theorem \ref{optimal}). In section \ref{sec4}, we prove the first order necessary conditions optimality for the optimal control problem via Pontryagin's maximum principle (Theorem \ref{main}). We characterize the optimal control in terms of the adjoint variable. Furthermore,  in the same section, we show the uniqueness of optimal control in small time interval for the optimal control problem (Theorem \ref{thm4.5}).  For the critical case ($r=3$), the second order necessary and sufficient conditions of optimality for the optimal control problem is established in section \ref{sec5} (Theorems \ref{necessary} and \ref{sufficient}).  In the final section, we formulate a problem similar to that of the data assimilation problems of meteorology of  optimizing the initial data,  where we find the unknown optimal initial data by minimizing a suitable cost functional  subject to the 2D CBF equations. We find the first order necessary conditions as well as second order necessary and sufficient conditions of optimality (for the case $r=3$) for such problems in section \ref{sec6} (Theorems \ref{data}, \ref{thm6.3}, \ref{thm6.4}).

\section{Mathematical Formulation}\label{sec2}\setcounter{equation}{0}
This section is devoted for the necessary function spaces  needed to obtain the global solvability results of the system \eqref{1}-\eqref{1.2}. We mainly follow the work \cite{MTM7} for the functional framework.  We provide an abstract formulation of the system \eqref{1}-\eqref{1.2} and  discuss about some properties like local monotonicity, hemicontinuity, etc of the  linear and nonlinear operators. The existence of a unique global weak solution to the linearized problem corresponding to the system \eqref{1}-\eqref{1.2} is also established.

\subsection{Function spaces}\label{sub2.1} We denote $\C_0^{\infty}(\mathcal{O};\R^n)$ as the space of all infinitely differentiable functions  ($\R^n$-valued) with compact support in $\mathcal{O}\subset\R^n$.  We define 
$\mathcal{V}:=\{\u\in\C_0^{\infty}(\mathcal{O},\R^n):\nabla\cdot\u=0\},$
$\mathbb{H}$, $\mathbb{V}$ and $\wi\L^p$ as the the closure of $ \mathcal{V} $ in the Lebesgue space  $\L^2(\mathcal{O})=\mathrm{L}^2(\mathcal{O};\R^n),$
 in the Sobolev space $\H_0^1(\mathcal{O})=\mathrm{H}_0^1(\mathcal{O};\R^n),$
 in the Lebesgue space $ \L^p(\mathcal{O})=\mathrm{L}^p(\mathcal{O};\R^n),$
for $p\in(2,\infty)$, respectively. Then under some smoothness assumptions on the boundary (one can take $\C^2$ boundary), we characterize the spaces $\H$, $\V$ and $\widetilde{\L}^p$ as 
\begin{align*}
\H&=\{\u\in\L^2(\mathcal{O}):\nabla\cdot\u=0,\u\cdot\mathbf{n}\big|_{\partial\mathcal{O}}=0\},\\
\V&=\{\u\in\H_0^1(\mathcal{O}):\nabla\cdot\u=0\},\\
\widetilde{\L}^p&=\{\u\in\L^p(\mathcal{O}):\nabla\cdot\u=0, {\u\cdot\mathbf{n}\big|_{\partial\mathcal{O}}=0}\},
\end{align*}
where $\mathbf{n}$ is the unit outward normal to $\partial\mathcal{O}$,  and $\u\cdot\n\big|_{\partial\mathcal{O}}$ should be understood in the sense of trace in $\H^{-1/2}(\partial\mathcal{O})$ (cf. Theorem 1.2, Chapter 1, \cite{Te}). The norms in $\H,\V,\wi\L^p$ are denoted by $\|\u\|_{\H}^2:=\int_{\mathcal{O}}|\u(x)|^2\d x,
$, $\|\u\|_{\V}:=\int_{\mathcal{O}}|\nabla\u(x)|^2\d x$ and $\|\u\|_{\widetilde{\L}^p}^p=\int_{\mathcal{O}}|\u(x)|^p\d x$, respectively. Let $(\cdot,\cdot)$ denote the inner product in the Hilbert space $\H$ and $\langle \cdot,\cdot\rangle $ denote the induced duality between the spaces $\V$  and its dual $\V'$ as well as $\widetilde{\L}^p$ and its dual $\widetilde{\L}^{p'}$, where $\frac{1}{p}+\frac{1}{p'}=1$. Note that $\H$ can be identified with its own dual  and we have the Gelfand triple $\V\subset\H\subset\V'$ and the embedding $\V\subset\H$ is compact.  Let $\mathbb{U}$ be a separable Hilbert space identified with its own dual and we denote $\|\cdot\|_{\mathbb{U}}$ and $(\cdot,\cdot)_{\mathbb{U}}$ as the norm and inner product defined on $\mathbb{U}$.  We also denote $\mathcal{L}(\mathbb{U},\V')$ as the space of all bounded linear operator from $\mathbb{U}$ to $\V'$.  
\subsection{Linear operator}\label{sub2.2}
 Let $\mathcal{P}: \L^p(\mathcal{O}) \to\wi\L^p$, $p\in(1,\infty)$, denote the Helmholtz-Hodge projection (cf  \cite{DFHM,HKTY}). For $p=2$, it becomes an orthogonal projection.  We  define
\begin{equation*}
\left\{
\begin{aligned}
\A\u:&=-\mathcal{P}\Delta\u,\;\u\in\D(\A),\\ \D(\A):&=\V\cap\H^{2}(\mathcal{O}).
\end{aligned}
\right.
\end{equation*}
It can be easily seen that the operator $\A$ is a non-negative self-adjoint operator in $\H$ with $\V=\D(\A^{1/2})$ and \begin{align}\label{2.7a}\langle \A\u,\u\rangle =\|\u\|_{\V}^2,\ \textrm{ for all }\ \u\in\V, \ \text{ so that }\ \|\A\u\|_{\V'}\leq \|\u\|_{\V}.\end{align}
For a bounded domain $\mathcal{O}$, the operator $\A$ is invertible and its inverse $\A^{-1}$ is bounded, self-adjoint and compact in $\H$. 
Using spectral theorem, one can deduce  that 
\begin{align}\label{poin}
\|\nabla\u\|_{\mathbb{H}}^2=\langle \A\u,\u\rangle =\sum_{k=1}^{\infty}\uplambda_k|(\u,e_k)|^2\geq \lambda_1\sum_{k=1}^{\infty}|(\u,e_k)|^2=\lambda_1\|\u\|_{\mathbb{H}}^2,
\end{align}
for all $\u\in\V$, which is the well known \emph{Poincar\'e inequality}. 

\subsection{Bilinear operator}

Let us now define the trilinear form $b(\cdot,\cdot,\cdot):\V\times\V\times\V\to\R$ by $$b(\u,\v,\w)=\int_{\mathcal{O}}(\u(x)\cdot\nabla)\v(x)\cdot\w(x)\d x=\sum\limits_{i,j=1}^2\int_{\mathcal{O}}u_i(x)\frac{\partial v_j(x)}{\partial x_i}w_j(x)\d x.$$ If $\u, \v$ are such that the linear map $b(\u, \v, \cdot) $ is continuous on $\V$, the corresponding element of $\V'$ is denoted by $\B(\u, \v)$. We also denote $\B(\u) = \B(\u, \u)=\mathcal{P}(\u\cdot\nabla)\u$.
An integration by parts yields 
\begin{equation}\label{b0}
\left\{
\begin{aligned}
b(\u,\v,\v) &= 0,\text{ for all }\u,\v \in\V,\\
b(\u,\v,\w) &=  -b(\u,\w,\v),\text{ for all }\u,\v,\w\in \V.
\end{aligned}
\right.\end{equation}
The Sobolev inequality (Lemma 1, 2, Chapter 1, \cite{OAL})  implies $\V\subset\wi\L^4\subset\H\subset\wi\L^{4/3}\subset\V'$. Using H\"older's inequality, we deduce the following inequality:
\begin{align*}
|b(\u,\v,\w)|=|b(\u,\w,\v)|\leq \|\u\|_{\wi\L^4}\|\nabla\w\|_{\H}\|\v\|_{\wi\L^4},\text{ for all }\u,\v,\w\in\V,
\end{align*}
and hence  applying Ladyzhenskaya's inequality, we get 
\begin{align}\label{2.9a}
\|\mathrm{B}(\u,\v)\|_{\V'}\leq \|\u\|_{\wi\L^4}\|\v\|_{\wi\L^4}\leq \sqrt{2}\|\u\|_{\H}^{1/2}\|\u\|_{\V}^{1/2}\|\v\|_{\H}^{1/2}\|\v\|_{\V}^{1/2}, \ \text{ for all} \ \u,\v\in\V. 
\end{align}
For $\u,\v\in\V$, we further have 
\begin{align}\label{2p8}
\langle\B(\u)-\B(\v),\u-\v\rangle&=\langle\B(\u-\v,\v),\u-\v\rangle \leq\|\v\|_{\V}\|\u-\v\|_{\wi\L^4}^2,
\end{align}
using H\"older's inequality.

\subsection{Nonlinear operator}\label{sub2.4}
Let us now consider the operator $\mathcal{C}(\u):=\mathcal{P}(|\u|^{r-1}\u)$. It is immediate that $\langle\mathcal{C}(\u),\u\rangle =\|\u\|_{\widetilde{\L}^{r+1}}^{r+1}$. Furthermore, for all $\u\in\wi\L^{r+1}$, the map is Gateaux differentiable with Gateaux derivative 
{
\begin{align}\label{29}
\mathcal{C}'(\u)\v&=\left\{\begin{array}{cc}\mathcal{P}(\v),&\text{ for }r=1,\\ \left\{\begin{array}{cc}\mathcal{P}(|\u|\v)+\mathcal{P}\left(\frac{\u}{|\u|}(\u\cdot\v)\right),&\text{ if }\u\neq \mathbf{0},\\\mathbf{0},&\text{ if }\u=\mathbf{0},\end{array}\right.&\text{ for } r=2,\\ \mathcal{P}(|\u|^{r-1}\v)+(r-1)\mathcal{P}(\u|\u|^{r-3}(\u\cdot\v)), &\text{ for }r\geq 3,\end{array}\right.
\end{align}}
for all $\v\in\widetilde{\L}^{r+1}$.  For $\u,\v\in\wi\L^{r+1}$, it can be easily seen that 
\begin{align}\label{2.9}
\langle\mathcal{C}'(\u)\v,\v\rangle=\int_{\mathcal{O}}|\u(x)|^{r-1}|\v(x)|^2\d x+(r-1)\int_{\mathcal{O}}|\u(x)|^{r-3}|\u(x)\cdot\v(x)|^2\d x\geq 0,
\end{align}
for $r\geq 1$ (note that for the case $r=2$, \eqref{2.9} holds, since in that case the second integral becomes $\int_{\{x\in\mathcal{O}:\u(x)\neq 0\}}\frac{1}{|\u(x)|}|\u(x)\cdot\v(x)|^2\d x\geq 0$). For $r\geq 3$, we further have 
{
\begin{align}\label{30}
\mathcal{C}''(\u)(\v\otimes\w)&=(r-1)\mathcal{P}\left\{|\u|^{r-3}\left[(\u\cdot\w)\v+(\u\cdot\v)\w+(\w\cdot\v)\u\right]\right\}\nonumber\\&\quad+(r-1)(r-3)\mathcal{P}\left[|\u|^{r-5}(\u\cdot\v)(\u\cdot\w)\u\right],
\end{align}}
for all $\u,\v,\w\in\widetilde{\L}^{r+1}$ (for $r=4$, the final term appearing in the right hand side of the equality \eqref{30} is defined for $\u\neq\mathbf{0}$ and one has to take $\mathbf{0}$ for $\u=\mathbf{0}$). For $0<\theta<1$ and $\u,\v\in\widetilde{\L}^{r+1}$, using Taylor's formula (Theorem 6.5, \cite{RFC}, \cite{PGC}), we find
\begin{align}
\langle \mathcal{P}(|\u|^{r-1}\u)-\mathcal{P}(|\v|^{r-1}\v),\w\rangle&\leq \|(|\u|^{r-1}\u)-(|\v|^{r-1}\v)\|_{\widetilde{\L}^{\frac{r+1}{r}}}\|\w\|_{\widetilde{\L}^{r+1}}\nonumber\\&\leq \sup_{0<\theta<1}r\|\theta\u+(1-\theta)\v\|_{\widetilde{\L}^{r+1}}^{r-1}\|\u-\v\|_{\widetilde{\L}^{r+1}}\|\w\|_{\widetilde{\L}^{r+1}}\nonumber\\&\leq r\left(\|\u\|_{\widetilde{\L}^{r+1}}+\|\v\|_{\widetilde{\L}^{r+1}}\right)^{r-1}\|\u-\v\|_{\widetilde{\L}^{r+1}}\|\w\|_{\widetilde{\L}^{r+1}},
\end{align}
for all $\u,\v,\w\in\widetilde{\L}^{r+1}$. 
For any $r\in[1,\infty)$, from \cite{MTM7}, we infer that 
\begin{align}\label{2.23}
&\langle\mathcal{P}(\u|\u|^{r-1})-\mathcal{P}(\v|\v|^{r-1}),\u-\v\rangle\geq \frac{1}{2}\||\u|^{\frac{r-1}{2}}(\u-\v)\|_{\H}^2+\frac{1}{2}\||\v|^{\frac{r-1}{2}}(\u-\v)\|_{\H}^2\geq 0.
\end{align}
	It is important to note that 
\begin{align}\label{a215}
\|\u-\v\|_{\wi\L^{r+1}}^{r+1}\leq 2^{r-2}\||\u|^{\frac{r-1}{2}}(\u-\v)\|_{\L^2}^2+2^{r-2}\||\v|^{\frac{r-1}{2}}(\u-\v)\|_{\L^2}^2,
\end{align}
for $r\geq 1$ (replace $2^{r-2}$ with $1,$ for $1\leq r\leq 2$). 
\begin{remark}\label{rem2.2}
{Even though we have defined the the operator $\mathcal{C}(\cdot)$ for $r\in[1,\infty)$, we consider the cases $r=1,2,3$ only in the rest of the paper. For $r=3$, from \eqref{30}, we deduce that
\begin{align*}
\mathcal{C}''(\u)(\v\otimes\w)&=2\mathcal{P}\left\{\left[(\u\cdot\w)\v+(\u\cdot\v)\w+(\w\cdot\v)\u\right]\right\},\\
\mathcal{C}'''(\u)(\v\otimes\w\otimes\x)&=2\mathcal{P}\left\{\left[(\x\cdot\w)\v+(\x\cdot\v)\w+(\w\cdot\v)\x\right]\right\},
\end{align*}
for all $\u,\v,\w,\x\in\wi\L^4$. }
\end{remark}
\subsection{Local monotonicity}
Let us now discuss about the local monotonicity property of the linear and nonlinear operators, which plays a crucial role in the global solvability of the system \eqref{1}. 
\begin{theorem}[Theorem 2.11, \cite{MTM1}, Remark 2.4, \cite{MTM7}]\label{thm2.2}
	Let $\u,\v\in\V$. Then,	for the operator $\G(\u)=\mu \A\u+\B(\u)+\alpha\u+\beta\mathcal{C}(\u)$, we  have 
	\begin{align}\label{fe2}
	\langle(\G(\u)-\G(\v),\u-\v\rangle+ \frac{27}{32\mu ^3}N^4\|\u-\v\|_{\H}^2\geq 0,
	\end{align}
	for all $\v\in{\mathbb{B}}_N$, where ${\mathbb{B}}_N$ is an $\widetilde{\L}^4$-ball of radius $N$, that is,
	$
	{\mathbb{B}}_N:=\big\{\z\in\widetilde{\L}^4:\|\z\|_{\widetilde{\L}^4}\leq N\big\}.
	$
\end{theorem}

\begin{theorem}[Theorem 2.3, \cite{MTM7}]\label{thm2.3}
	For the critical case $r=3$ with $2\beta\mu \geq 1$, the operator $\G(\cdot):\V\to \V'$ is globally monotone, that is, for all $\u,\v\in\V$, we have 
	\begin{align}\label{218}\langle\G(\u)-\G(\v),\u-\v\rangle\geq 0.\end{align}
\end{theorem}

\subsection{Abstract formulation and global solvability}
On taking  the Helmholtz-Hodge projection $\mathcal{P}$ onto the system \eqref{1}, we obtain 
\begin{equation}\label{2.18}
\left\{
\begin{aligned}
\partial_t\u(t)+\mu\A\u(t)+\B(\u(t))+\alpha\u(t)+\beta\mathcal{C}(\u(t))&=\f(t)+\D\U(t),\ \text{ in }\ \V', \\
\u(0)&=\u_0,
\end{aligned}
\right.
\end{equation}
a.e. $t\in[0,T]$, where $\D\in\mathcal{L}(\mathbb{U},\V')$ and for simplicity, we used the notation $\mathcal{P}\f=\f$ and $\mathcal{P}\D\U=\D\U$. 
\begin{definition}\label{weakd}
Let $\u_0\in\H$,  $\f\in\mathrm{L}^2(0,T;\V')$ and $\U\in\mathrm{L}^2(0,T;\mathbb{U})$ be given. 	For $r=1,2,3$, a function  $$\u\in\mathrm{L}^{\infty}([0,T];\H)\cap\mathrm{L}^2(0,T;\V),$$  with $\partial_t\u\in\mathrm{L}^{p}(0,T;\mathbb{V}'),$ where $p=2$, for $r=1,2$ and $p=4/3$, for $r=3$,  is called a \emph{weak solution} to the system (\ref{2.18}), if for  $\u_0\in\H$ and $\v\in\V$, $\u(\cdot)$ satisfies:
	\begin{equation}
	\left\{
	\begin{aligned}
	\langle\partial_t\u(t)+\mu \mathrm{A}\u(t)+\mathrm{B}(\u(t))+\beta\mathcal{C}(\u(t)),\v\rangle&=\langle\f(t),\v\rangle+\langle\D\U(t),\v\rangle,\\
	\lim_{t\downarrow 0}\int_{\mathcal{O}}\u(t,x)\v(x)\d x&=\int_{\mathcal{O}}\u_0(x)\v(x)\d x,
	\end{aligned}
	\right.
	\end{equation}
	and the energy equality:
	\begin{align}\label{2p37}
&	\|\u(t)\|^2_{\H}+2\mu \int_0^t\|\u(s)\|^2_{\V}\d s+2\alpha\int_0^t\|\u(s)\|_{\H}^2\d s+2\beta\int_0^t\|\u(s)\|_{\widetilde{\L}^{r+1}}^{r+1}\d s\nonumber\\&=\|\u_0\|_{\H}^2+2\int_0^t\langle\f(s),\u(s)\d s+2\int_0^t\langle\D\U(s),\u(s)\rangle\d s,
	\end{align}
	for all $t\in[0,T]$. 
\end{definition}
\begin{remark}
	From the energy equality \eqref{2p37}, one can easily deduce that 
	\begin{align}\label{p36}
	&	\|\u(t)\|_{\H}^2+\mu\int_0^t\|\u(s)\|_{\V}^2\d s+2\alpha\int_0^t\|\u(s)\|_{\H}^2\d s+2\beta\int_0^t\|\u(s)\|_{\wi\L^{r+1}}^{r+1}\d s\nonumber\\&\leq\|\u_0\|_{\H}^2+\frac{2}{\mu}\int_0^T\|\f(t)\|_{\V'}^2\d t+\frac{2}{\mu}\|\D\|_{\mathcal{L}(\mathbb{U},\V')}^2\int_0^T\|\U(t)\|_{\mathbb{U}}^2\d t=:K_T,
	\end{align}
for all $t\in[0,T]$.	Note that $\partial_t\u\in\mathrm{L}^2(0,T;\V'),$ for $r=2$ and $\partial_t\u\in\mathrm{L}^{2}(0,T;\V')+\mathrm{L}^{4/3}(0,T;\wi\L^{4/3}),$ for $r=3$ (so that $\partial_t\u\in\mathrm{L}^{4/3}(0,T;\V')$). Thus applying the Theorem 3, section 5.9.2, \cite{evans}, we know that $\u\in\mathrm{C}(0,T;\H)$, for $r=2$ and using the  Theorem 2, section 5.9.2, \cite{evans}, we obtain $\u\in\C([0,T];\V')$, for $r=3$. In fact, one can show that $\u\in\C([0,T];\H)$ satisfying the energy equality \eqref{2p37} for $r=3$,
	since  Theorem 4.1, \cite{GGP} (see Lemma 2.5, \cite{GGP} for approximations) is applicable as $\u\in\mathrm{L}^4(0,T;\wi\L^4)$, for $r=3$. 
\end{remark}

The local mononicity and hemicontinuity properties of the linear and nonlinear operators,  and  a localized version of the Minty-Browder technique imply the following global solvability results (similar methods were used in \cite{FS} for the global solvability of the 2D NSE and \cite{MTM1} for the existence and uniqueness of weak solutions of the 2D Oldroyd models of order one). 
\begin{theorem}[Theorem 3.4, \cite{MTM7}]\label{main2}
	Let $\u_0\in \H$   be given.  Then  there exists a unique weak solution to the system (\ref{2.18}) satisfying $$\u\in\mathrm{C}([0,T];\H)\cap\mathrm{L}^2(0,T;\V),  $$ and \eqref{p36}. 
\end{theorem}

\subsection{The linearized system} In this subsection, we consider  the following linearized system and discuss about its global solvability results under the assumption that $\u(\cdot)$ is \emph{a weak solution} to the system \eqref{2.18}. 
\begin{equation}\label{4}
\left\{
\begin{aligned}
\partial_t\mathfrak{w}(t)+\mu\A\mathfrak{w}(t)+\B'(\u(t))\mathfrak{w}(t)+\alpha\mathfrak{w}(t)+\beta\mathcal{C}'(\u(t))\mathfrak{w}(t)&=\g(t),\ \text{ in }\ \V',\\
\mathfrak{w}(0)&=\mathfrak{w}_0, 
\end{aligned}
\right.
\end{equation}
a.e. $t\in[0,T]$,  where  $\B'(\u)\mathfrak{w}=\B(\u,\mathfrak{w})+\B(\mathfrak{w},\u)$, $\mathcal{C}'(\u)\mathfrak{w}$ is defined in \eqref{29}, $\mathfrak{w}_0\in\H$ and $\g\in\mathrm{L}^2(0,T;\V')$.  We first provide an a-priori energy estimate satisfied by the system \eqref{4} formally. A rigorous justification can be given using a standard Faedo-Galerkin's approximation technique. Taking the inner product with $\mathfrak{w}(\cdot)$ to the first equation in \eqref{4}, we obtain 
\begin{align}\label{2p16}
\frac{1}{2}\frac{\d }{\d t}\|\mathfrak{w}(t)\|_{\H}^2&+\mu\|\mathfrak{w}(t)\|_{\V}^2+\alpha\|\mathfrak{w}(t)\|_{\H}^2+\beta\||\u(t)|^{\frac{r-1}{2}}\mathfrak{w}(t)\|_{\H}^2\d t\nonumber\\&\quad+(r-1)\beta\||\u(t)|^{\frac{r-3}{2}}(\u(t)\cdot\mathfrak{w}(t))\|_{\H}^2\nonumber\\&=-\langle\B(\mathfrak{w}(t),\u(t)),\mathfrak{w}(t)\rangle+\langle\g(t),\mathfrak{w}(t)\rangle\nonumber\\&\leq\|\u(t)\|_{\V}\|\mathfrak{w}(t)\|_{\wi\L^4}^2+\|\g(t)\|_{\V'}\|\mathfrak{w}(t)\|_{\V}\nonumber\\&\leq\frac{\mu}{2}\|\mathfrak{w}(t)\|_{\V}^2+\frac{2}{\mu}\|\u(t)\|_{\V}^2\|\mathfrak{w}(t)\|_{\H}^2+\frac{1}{\mu}\|\g(t)\|_{\V'}^2,
\end{align}
where we used H\"older's, Ladyzhenskaya's and Young's inequalities. Integrating the above inequality from $0$ to $t$, we find
\begin{align}\label{2p18}
&\|\mathfrak{w}(t)\|_{\H}^2+\mu\int_0^t\|\mathfrak{w}(s)\|_{\V}^2\d s+2\alpha\int_0^t\|\mathfrak{w}(s)\|_{\H}^2\d s+2\beta\int_0^t\||\u(s)|^{\frac{r-1}{2}}\mathfrak{w}(s)\|_{\H}^2\d s\nonumber\\&\quad+2(r-1)\beta\int_0^t\||\u(s)|^{\frac{r-3}{2}}(\u(s)\cdot\mathfrak{w}(s))\|_{\H}^2\d s\nonumber\\&\leq\|\mathfrak{w}_0\|_{\H}^2+\frac{4}{\mu}\int_0^t\|\u(s)\|_{\V}^2\|\mathfrak{w}(s)\|_{\H}^2\d s+\frac{2}{\mu}\int_0^t\|\g(s)\|_{\V'}^2\d s, 
\end{align}
for all $t\in[0,T]$. An application of Gornwall's inequality in \eqref{2p18} yields 
\begin{align}\label{2p19}
&\|\mathfrak{w}(t)\|_{\H}^2\leq\left\{\|\mathfrak{w}_0\|_{\H}^2+\frac{2}{\mu}\int_0^T\|\g(t)\|_{\V'}^2\d t\right\}\exp\left(\frac{4}{\mu}\int_0^T\|\u(t)\|_{\V}^2\d t\right),
\end{align}
for all $t\in[0,T]$. Using \eqref{2p19} in \eqref{2p18}, we get 
\begin{align}\label{2p20}
&\sup_{t\in[0,T]}\|\mathfrak{w}(t)\|_{\H}^2+\mu\int_0^T\|\mathfrak{w}(t)\|_{\V}^2\d t+2\alpha\int_0^T\|\mathfrak{w}(t)\|_{\H}^2\d t+2\beta\int_0^T\||\u(t)|^{\frac{r-1}{2}}\mathfrak{w}(t)\|_{\H}^2\d t\nonumber\\&\quad+2(r-1)\beta\int_0^T\||\u(t)|^{\frac{r-3}{2}}(\u(t)\cdot\mathfrak{w}(t))\|_{\H}^2\d t\nonumber\\&\leq\left\{\|\mathfrak{w}_0\|_{\H}^2+\frac{2}{\mu}\int_0^T\|\g(t)\|_{\V'}^2\d t\right\}\exp\left(\frac{8}{\mu}\int_0^T\|\u(t)\|_{\V}^2\d t\right).
\end{align}

Let us now obtain the estimates for time derivatives.  For $r=2$ and $\v\in\mathrm{L}^2(0,T;\V)$, from the first equation in \eqref{4}, we find 
\begin{align}\label{2p13}
&\int_0^T|\langle\partial_t\mathfrak{w}(t),\v(t)\rangle|\d t\nonumber\\&=\left|\int_0^T\langle-\mu\A\mathfrak{w}(t)-\B'(\u(t))\mathfrak{w}(t)-\alpha\mathfrak{w}(t)-\beta\mathcal{C}'(\u(t))\mathfrak{w}(t)+\g(t),\v(t)\rangle\d t\right|\nonumber\\&\leq \mu\int_0^T\|\mathfrak{w}(t)\|_{\V}\|\v(t)\|_{\V}\d t+\int_0^T\|\u(t)\|_{\wi\L^4}\|\v(t)\|_{\V}\|\mathfrak{w}(t)\|_{\wi\L^4}\d t+\int_0^T\|\mathfrak{w}(t)\|_{\wi\L^4}\|\v(t)\|_{\V}\|\u(t)\|_{\wi\L^4}\d t\nonumber\\&\quad+\alpha\int_0^T\|\mathfrak{w}(t)\|_{\H}\|\v(t)\|_{\H}\d t+2\beta\int_0^T\|\u(t)\|_{\wi\L^4}\|\mathfrak{w}(t)\|_{\wi\L^4}\|\v(t)\|_{\H}+\int_0^T\|\g(t)\|_{\V'}\|\v(t)\|_{\V}\d t\nonumber\\&\leq\bigg\{\mu\left(\int_0^T\|\mathfrak{w}(t)\|_{\V}^2\d t\right)^{1/2}+2(1+\beta)\left(\int_0^T\|\u(t)\|_{\wi\L^4}^2\|\mathfrak{w}(t)\|_{\wi\L^4}^2\d t\right)^{1/2}+\frac{\alpha}{\sqrt{\lambda_1}}\left(\int_0^T\|\mathfrak{w}(t)\|_{\H}^2\d t\right)^{1/2}\nonumber\\&\qquad+\left(\int_0^T\|\g(t)\|_{\V'}^2\d t\right)^{1/2}\bigg\}\left(\int_0^T\|\v(t)\|_{\V}^2\d t\right)^{1/2}, 
\end{align}
so that we get $\|\partial_t\mathfrak{w}\|_{\mathrm{L}^2(0,T;\V')}<\infty$. Thus, it is a consequence of Theorem  3, section 5.9.2, \cite{evans} that $\u^* \in \mathrm{C}([0,T];\H) $.

For $r=3$ and $\v\in\mathrm{L}^2(0,T;\V)\cap\mathrm{L}^4(0,T;\wi\L^4),$ the estimate given in \eqref{2p13} holds true except for the estimate for the term involving $\mathcal{C}'(\u)\mathfrak{w}$. It can be easily seen that 
\begin{align}
\left|\int_0^T\langle\mathcal{C}'(\u(t))\mathfrak{w}(t),\v(t)\rangle\d t\right|&\leq r \int_0^T\|\u(t)\|_{\wi\L^4}^2\|\mathfrak{w}(t)\|_{\wi\L^4}\|\v(t)\|_{\wi\L^4}\d t\nonumber\\&\leq r \left(\int_0^T\|\u(t)\|_{\wi\L^4}^{8/3}\|\mathfrak{w}(t)\|_{\wi\L^4}^{4/3}\d t\right)^{3/4}\left(\int_0^T\|\v(t)\|_{\wi\L^4}^4\d t\right)^{1/4}, 
\end{align}
so that we find 
\begin{align}\label{2p25}
\left(\int_0^T\|\mathcal{C}'(\u(t))\mathfrak{w}(t)\|_{\wi\L^{4/3}}^{4/3}\d t\right)^{3/4}\leq \left(\int_0^T\|\u(t)\|_{\wi\L^4}^4\d t\right)^{1/2}\left(\int_0^T\|\mathfrak{w}(t)\|_{\wi\L^4}^4\d t\right)^{1/4}<\infty, 
\end{align}
and hence $\partial_t\mathfrak{w}\in\mathrm{L}^2(0,T;\V')+\mathrm{L}^{4/3}(0,T;\wi\L^{4/3})$. But we know that $\V\subset\wi\L^4\subset\H\subset\L^{4/3}\subset\V'$, and thus we get $\partial_t\mathfrak{w}\in\mathrm{L}^{4/3}(0,T;\V')$. Thus an application of the Theorem 2, section 5.9.2, \cite{evans} yields $\mathfrak{w}\in\C([0,T];\V')$. The reflexivity of the space $\H$ and  Proposition 1.7.1 \cite{PCAM} gives $\mathfrak{w}\in\C_w([0,T];\H)$ and the map $t\mapsto\|\mathfrak{w}(t)\|_{\H}$ is bounded, where $\C_w([0,T];\H)$ represents the space of functions $\mathfrak{w}:[0,T]\to\H$ which are weakly continuous.  Since $\mathfrak{w}\in\mathrm{L}^4(0,T;\wi\L^4)$, one can show that $\mathfrak{w}\in\C([0,T];\H)$ satisfying the energy equality by following similar arguments as in Theorem 4.1, \cite{GGP}.  Note that $\mathfrak{w}\in\mathrm{L}^2(0,T;\V)\subset\mathrm{L}^{4/3}(0,T;\V)$ and $\partial_t\mathfrak{w}\in\mathrm{L}^{4/3}(0,T;\V')$ implies that the embedding of $\{\mathfrak{w}\in\mathrm{L}^{4/3}(0,T;\V):\partial_t\mathfrak{w}\in\mathrm{L}^{4/3}(0,T;\V')\}$ into $\mathrm{L}^{4/3}(0,T;\H)$ is compact due to an application of the Aubin-Lions compactness theorem (see \cite{lions1}, Theorem 1, \cite{SJ}). The existence of weak solutions to the system \eqref{4} can be proved using a standard Faedo-Galerkin's approximation technique and we have the following Theorem. Due to the linearity of the system \eqref{4}, uniqueness of weak solutions  follows immediately from the estimate \eqref{2p18}. 

\begin{theorem}\label{linear}
	Let  $\mathfrak{w}_0\in\mathbb{H}$ and $\f\in\mathrm{L}^2(0,T;\V')$ be given.  Then,
	there exists \emph{a unique weak solution} to the system \eqref{4} satisfying 
	$$\mathfrak{w}\in\mathrm{L}^{\infty}(0,T;\H)\cap\mathrm{L}^2(0,T;\V), $$ with $\partial_t\mathfrak{w}\in\mathrm{L}^{p}(0,T;\mathbb{V}'),$ where $p=2$, for $r=1,2$ and $p=4/3$, for $r=3$ (so that $\mathfrak{w}\in\C([0,T];\H)$),  and the energy equality 
	\begin{align*}
&	\|\mathfrak{w}(t)\|_{\H}^2+2\mu\int_0^t\|\mathfrak{w}(s)\|_{\V}^2\d s+2\alpha\int_0^t\|\mathfrak{w}(s)\|_{\H}^2\d s+2\beta\int_0^t\||\u(s)|^{\frac{r-1}{2}}\mathfrak{w}(s)\|_{\H}^2\d s\nonumber\\&\quad+2(r-1)\beta\int_0^t\||\u(s)|^{\frac{r-3}{2}}(\u(s)\cdot\mathfrak{w}(s))\|_{\H}^2\d s\nonumber\\&=\|\mathfrak{w}_0\|_{\H}^2-2\int_0^t\langle\B(\mathfrak{w}(s),\u(s)),\mathfrak{w}(s)\rangle\d s+2\int_0^t\langle\g(s),\mathfrak{w}(s)\rangle\d s,
	\end{align*}
	for all $t\in[0,T]$. 
\end{theorem}

\section{Optimal Control Problem: Existence}\label{sec3}\setcounter{equation}{0}In this section, we formulate a distributed optimal control problem  as the minimization of a suitable cost functional subject to the controlled 2D CBF system \eqref{2.18}. The main objective of this section is to prove the existence of an optimal control that minimizes the cost functional 
\begin{equation}\label{cost}
\begin{aligned}
\mathscr{J}(\u,\U) := \frac{1}{2} \int_0^T\|\u(t)-\u_d\|_{\H}^2 \d t+\frac{1}{2}\int_0^T\|\nabla\times\u(t)\|_{\H}^2\d t+ \int_0^Th(\U(t))\d t+\frac{1}{2}\|\u(T)-\u^f\|_{\H}^2,
\end{aligned}
\end{equation}
subject to the constraint  \eqref{2.18}.  In \eqref{cost}, $\u_d\in\mathrm{L}^2(0,T;\H)$ denotes the desired velocity field (or the reference velocity) and $\u^f\in\H$ represents the desired velocity at time $T$.  On the physical background, our goal is to determine the control $\U$ in such a way that the velocity vector is as close as possible, in the sense of \eqref{control problem} (see below), to the desired	velocity $\u_d\in\mathrm{L}^2(0,T;\H)$ and the turbulence is minimal (enstrophy minimization). Since $\u(\cdot)$ is divergence free, we know that $\|\mathrm{curl \ }\u\|_{\H}=\|\nabla\times\u\|_{\H}=\|\nabla\u\|_{\H}=\|\A^{1/2}\u\|_{\H}$ and we replace $\|\nabla\times\u\|_{\H}$ with $\|\nabla\u\|_{\H}$ in \eqref{cost} in the rest of the paper. In the sequel, we impose the following assumption on  $h(\cdot)$. 
\begin{hypothesis}\label{hyp}
The function $h:\mathbb{U}\to(-\infty,\infty)$ is convex and  lower semi-continuous. Moreover, it satisfies the coercivity condition: \begin{align}\label{2a19}h(\U)\geq \kappa_1\|\U\|_{\mathbb{U}}^2+\kappa_2, \ \text{ for all } \ \U\in\mathbb{U}, \end{align} for some $\kappa_1>0$ and $\kappa_2\in\mathbb{R}$. 

 \end{hypothesis}

 We take the set of all admissible control class $\mathscr{U}_{\mathrm{ad}}=\mathrm{L}^{2}(0,T;\mathbb{U})$.  Next, we  provide the definition of admissible class of solutions.

\begin{definition}[Admissible class]\label{definition 1}
	The \emph{admissible class} $\mathscr{A}_{\mathrm{ad}}$ of pairs $$(\u,\U)\in\mathrm{C}([0,T];\H)\cap\mathrm{L}^2(0,T;\V')\times\mathrm{L}^{2}(0,T;\mathbb{U}) $$ is defined as the set of states $\u(\cdot)$ solving the system \eqref{2.18} with control $\U \in \mathscr{U}_{ad}$. That is,
	\begin{align*}
	\mathscr{A}_{\mathrm{ad}}&:=\big\{(\u,\U) :\u\in\mathrm{C}([0,T];\H)\cap\mathrm{L}^2(0,T;\V) \text{ is \text{the unique weak solution} of }\eqref{2.18} \nonumber\\&\qquad \text{ with the control }\U\in\mathrm{L}^{2}(0,T;\mathbb{U})\big\}.
	\end{align*}
\end{definition}
Clearly $\mathscr{A}_{\mathrm{ad}}$ is a nonempty set as for any $\U \in \mathscr{U}_{\mathrm{ad}}$, there exists \emph{a unique weak solution} of the system \eqref{2.18}. In view of the above definition, the optimal control problem we are considering can be  formulated as:
\begin{align}\label{control problem}\tag{P}
\min_{ (\u,\U) \in \mathscr{A}_{\mathrm{ad}}}  \mathscr{J}(\u,\U).
\end{align}
A solution to the problem \eqref{control problem} is called an \emph{optimal solution}. The optimal pair is denoted by $(\u^* , \U^*)$ and the control $\U^*$ is called an \emph{optimal control}.

\subsection{The adjoint system} In order to establish the first order necessary conditions of optimality, we need to find the adjoint system corresponding to \eqref{2.18}. Remember that optimal control is characterized via adjoint variable. Let the adjoint variable $\p(\cdot)$ satisfies the following adjoint system: 
\begin{eqnarray}\label{adj}
\left\{
\begin{aligned}
-\partial_t\p(t)+\mu\A\p(t)+(\B'(\u(t)))^*\p (t)+\alpha\p(t)+\beta\mathcal{C}'(\u(t))\p(t)&=\h(t), \ \text{ in }\ \V', \\ \p(T)&=\p_T,
\end{aligned}
\right.
\end{eqnarray}
a.e. $t\in[0,T]$,  where $\langle(\B'(\u))^*\p,\mathbf{q}\rangle=\langle\p,\B'(\u)\mathbf{q}\rangle=\langle\B(\u,\mathbf{q}),\p\rangle+\langle\B(\mathbf{q},\u),\p\rangle$ and $\h\in\mathrm{L}^2(0,T;\V')$. Let us now obtain an a-priori energy estimate satisfied by $\p$. Taking the inner product with $\p(\cdot)$ to the first equation in \eqref{adj}, we find
\begin{align}\label{37}
-\frac{1}{2}\frac{\d}{\d t}\|\p(t)\|_{\H}^2&+\mu\|\p(t)\|_{\V}^2+\alpha\|\p(t)\|_{\H}^2+\beta\||\u(t)|^{\frac{r-1}{2}}\p(t)\|_{\H}^2\nonumber\\&\quad+(r-1)\beta\||\u(t)|^{\frac{r-3}{2}}(\u(t)\cdot\p(t))\|_{\H}^2\nonumber\\&=\langle\B(\p(t),\u(t)),\p(t)\rangle+\langle\h(t),\p(t)\rangle\nonumber\\&\leq\|\u(t)\|_{\V}\|\p(t)\|_{\wi\L^4}^2+\|\h(t)\|_{\V'}\|\p(t)\|_{\V}\nonumber\\&\leq\frac{\mu}{2}\|\p(t)\|_{\V}^2+\frac{2}{\mu}\|\u(t)\|_{\V}^2\|\p(t)\|_{\H}^2+\frac{1}{\mu}\|\h(t)\|_{\V'}^2,
\end{align}
for a.e. $t\in[0,T]$. Integrating the above inequality from $t$ to $T$, we obtain 
\begin{align}\label{39}
&\|\p(t)\|_{\H}^2+\mu\int_t^T\|\p(s)\|_{\V}^2\d s+2\alpha\int_t^T\|\p(s)\|_{\H}^2\d s+2\beta\int_t^T\||\u(s)|^{\frac{r-1}{2}}\p(s)\|_{\H}^2\d s\nonumber\\&\quad+2(r-1)\beta\int_t^T\||\u(s)|^{\frac{r-3}{2}}(\u(s)\cdot\p(s))\|_{\H}^2\d s\nonumber\\&\leq\|\p(T)\|_{\H}^2+\frac{4}{\mu}\int_t^T\|\u(s)\|_{\V}^2\|\p(s)\|_{\H}^2\d s +\frac{2}{\mu} \int_t^T\|\h(s)\|_{\V'}^2\d s.
\end{align}
An application of Gronwall's inequality in \eqref{39} yields 
\begin{align}\label{40}
\|\p(t)\|_{\H}^2&\leq \left(\|\p_T\|_{\H}^2+\frac{2}{\mu} \int_0^T\|\h(t)\|_{\V'}^2\d t\right)\exp\left(\frac{4}{\mu}\int_0^T\|\u(t)\|_{\V}^2\d t\right),
\end{align}
for all $t\in[0,T]$. 
Since  $\u\in\mathrm{L}^2(0,T;\V)$, the right hand side of the estimate in \eqref{40} is uniformly bounded. In order to obtain the time derivative estimate, for $r=2$ and  $\v\in\mathrm{L}^2(0,T;\V')\cap\mathrm{L}^4(0,T;\wi\L^4)$, we consider 
\begin{align}
&\int_0^T|\langle\partial_t\p(t),\v(t)\rangle|\d t\nonumber\\&\leq\mu\int_0^T|\langle\A\p(t),\v(t)\rangle|\d t+\int_0^T|\langle\p(t),\B(\u(t),\v(t))\rangle|\d t+\int_0^T|\langle\p(t),\B(\v(t),\u(t))\rangle|\d t\nonumber\\&\quad+\alpha\int_0^T|\langle\p(t),\v(t)\rangle|\d t+\beta\int_0^T\langle\mathcal{C}'(\u(t))\p(t),\v(t)\rangle|\d t+\int_0^T|\langle\h(t),\v(t)\rangle|\d t\nonumber\\&\leq\bigg\{\left(\mu+\frac{\alpha}{\sqrt{\lambda_1}}\right)\left(\int_0^T\|\p(t)\|_{\V}^2\d t\right)^{1/2}+\left(1+\frac{2\beta}{\sqrt{\lambda_1}}\right)\left(\int_0^T\|\p(t)\|_{\wi\L^4}^4\d t\right)^{1/4}\left(\int_0^T\|\u(t)\|_{\wi\L^4}^4\d t\right)^{1/4}\nonumber\\&\quad+\left(\int_0^T\|\h(t)\|_{\V'}^2\d t\right)^{1/2}\bigg\}\left(\int_0^T\|\v(t)\|_{\V}^2\d t\right)^{1/2}+\left(\int_0^T\|\p(t)\|_{\wi\L^4}^4\d t\right)^{1/4}\left(\int_0^T\|\u(t)\|_{\V}^2\d t\right)^{1/2}\nonumber\\&\qquad\times\left(\int_0^T\|\v(t)\|_{\wi\L^4}^4\d t\right)^{1/4},
\end{align}
so that $\partial_t\p\in\mathrm{L}^2(0,T;\V')+\mathrm{L}^{\frac{4}{3}}(0,T;\wi\L^{\frac{4}{3}})$. For $r=3$, we just need to estimate the term $\int_0^T|\langle\mathcal{C}'(\u(t))\p(t),\v(t)\rangle|\d t$. An application of H\"older's inequality yields 
\begin{align}
\int_0^T|\langle\mathcal{C}'(\u(t))\p(t),\v(t)\rangle|\d t&\leq 3\int_0^T\|\u(t)\|_{\wi\L^4}^2\|\p(t)\|_{\wi\L^4}\|\v(t)\|_{\wi\L^4}\d t\nonumber\\&\leq 3\left(\int_0^T\|\u(t)\|_{\wi\L^4}^4\d t\right)^{1/2}\left(\int_0^T\|\p(t)\|_{\wi\L^4}^4\d t\right)^{1/4}\left(\int_0^T\|\v(t)\|_{\wi\L^4}^4\d t\right)^{1/4},
\end{align}
so that $\mathcal{C}'(\u)\p\in\mathrm{L}^2(0,T;\V')+\mathrm{L}^{\frac{4}{3}}(0,T;\wi\L^{\frac{4}{3}})$. Thus, in both cases,  we obtain  $\|\partial_t\p\|_{\mathrm{L}^{\frac{4}{3}}(0,T;\V')}$ is  uniformly bounded.  Once again using a Faedo-Galerkin approximation technique, one can obtain the global solvability resuslts of the system \eqref{adj}. The following Theorem gives the global  existence and uniqueness of weak solution to the system (\ref{adj}) satisfying the energy equality.

\begin{theorem}\label{thm3.4}
For $\h\in\mathrm{L}^2(0,T;\V')$,
	there exists a unique weak solution to the system \eqref{adj} satisfying 
	$$\p\in\mathrm{L}^{\infty}(0,T;\H)\cap\mathrm{L}^2(0,T;\V), $$ 
with $\partial_t\p\in\mathrm{L}^{p}(0,T;\mathbb{V}'),$ where $p=2$, for $r=1$ and $p=4/3$, for $r=2,3$ (so that $\p\in\C([0,T];\H)$),	and 
	\begin{align*}
&	\|\p(t)\|_{\H}^2+2\mu\int_t^T\|\p(s)\|_{\V}^2\d s+2\alpha\int_t^T\|\p(s)\|_{\H}^2\d s+2\beta\int_t^T\||\u(s)|^{\frac{r-1}{2}}\p(s)\|_{\H}^2\d s\nonumber\\&\quad+2(r-1)\beta\int_t^T\||\u(s)|^{\frac{r-3}{2}}(\u(s)\cdot\p(s))\|_{\H}^2\d s\nonumber\\&=\|\p_T\|_{\H}^2-2\int_t^T\langle\B(\p(s),\u(s)),\p(s)\rangle\d s-2\int_t^T\langle\h(s),\p(s)\rangle\d s, 
	\end{align*}
	for all $t\in[0,T]$. 
\end{theorem}

\subsection{Existence of an optimal control}

Our next aim is to show that an optimal pair $(\u^*,\U^*)$ exists for the problem \eqref{control problem}.
\begin{theorem}[Existence of an optimal pair]\label{optimal}
	Let  $\mathbf{u}_0\in\H$ and  $\f\in\mathrm{L}^2(0,T;\V')$ be given.  Then there exists at least one pair $(\u^*,\U^*)\in\mathscr{A}_{\mathrm{ad}}$  such that the functional $ \mathscr{J}(\u,\U)$ attains its minimum at $(\u^*,\U^*)$, where $\u^*$ is the unique weak solution of the system  \eqref{2.18}  with the control $\U^*\in\mathscr{U}_{\mathrm{ad}}$.
\end{theorem}
\begin{proof}
	Let us first define $\mathscr{J} := \inf \limits _{\U \in \mathscr{U}_{\mathrm{ad}}}\mathscr{J}(\u,\U).$
	Since, $0\leq \mathscr{J} < +\infty$, there exists a minimizing sequence $\{\U_n\} \in \mathscr{U}_{\mathrm{ad}}$ such that $\lim\limits_{n\to\infty}\mathscr{J}(\u_n,\U_n) = \mathscr{J},$ where $\u_n$ is the unique weak solution of the system \eqref{2.18} with the control $\U_n$ and also the initial data  
$
	\u_n(0)=	\u_0 \in\H.
$
	Therefore, we have 	\begin{align}\label{38}
	\mathscr{J}\leq\mathscr{J}(\u_n,\U_n)\leq \mathscr{J}+\frac{1}{n},
\end{align}
	where 
	\begin{equation}
	\left\{
	\begin{aligned}
	\partial_t\u_n(t)+\mu\A\u_n(t)+\B(\u_n(t))+\alpha\u_n(t)+\beta\mathcal{C}(\u_n(t))&=\f(t)+\D\U_n(t),\ \text{ in } \ \V', \\
	\u_n(0)&=\u_0,
	\end{aligned}
	\right.
	\end{equation}
for a.e. $t\in[0,T]$.	From the condition \eqref{38},  we can find a large constant $\widetilde{C}>0,$ such that
	$$ \int_0^T \|\U_n(t)\|^2_{\mathbb{U}} \d t \leq  \widetilde{C} < +\infty .$$
	That is, the sequence $\{\U_n\}$ is uniformly bounded in the space $\mathrm{L}^2(0,T;\mathbb{U})$. Since $\u_n$ is the unique weak solution of the system \eqref{2.18} with control the $\U_n$, from the energy estimate \eqref{p36}, we have 
	\begin{align*}
	&\|\u_n(t)\|_{\H}^2+\mu\int_0^t\|\u_n(s)\|_{\V}^2\d s+2\alpha\int_0^t\|\u_n(s)\|_{\H}^2\d s+2\beta\int_0^t\|\u_n(s)\|_{\wi\L^{r+1}}^{r+1}\d s \nonumber\\&\leq\|\u_0\|_{\H}^2+\frac{2}{\mu}\int_0^t\|\f(s)\|_{\V'}^2\d s+\frac{2\widetilde{C}}{\mu}\|\D\|_{\mathcal{L}(\mathbb{U},\V')}^2,
	\end{align*}
	for all $t\in[0,T]$.	It can be easily seen  that the sequence $\{\u_n\} $ is uniformly bounded in $\mathrm{L}^{\infty}(0,T;\H)\cap \mathrm{L}^2(0,T;\V)\cap\mathrm{L}^{r+1}(0,T;\wi\L^{r+1})$. For $r=2$ and $\v\in\mathrm{L}^2(0,T;\V)$, we find 
	{
	\begin{align*}
	&	\int_0^T|\langle\partial_t\u_n(t),\v(t)\rangle|\d t\nonumber\\&\leq\mu\int_0^T|\langle\A\u_n(t),\v(t)\rangle|\d t+\int_0^T|\langle\B(\u_n(t)),\v(t)\rangle|\d t+\alpha\int_0^T|(\u_n(t),\v(t))|\d t\nonumber\\&\quad+\beta\int_0^T|\langle\mathcal{C}(\u_n(t)),\v(t)\rangle|\d t+\int_0^T|\langle\f(t),\v(t)\rangle|\d t+\int_0^T|\langle\D\U_{n_k}(t),\v(t)\rangle|\d t\nonumber\\&\leq\left\{\left(\mu+\frac{\alpha}{\sqrt{\lambda_1}}\right)\left(\int_0^T\|\u_n(t)\|_{\V}^2\d t\right)^{1/2}+\left(1+\frac{\beta}{\sqrt{\lambda_1}}\right)\left(\int_0^T\|\u_n(t)\|_{\wi\L^4}^4\d t\right)^{1/2}\right.\nonumber\\&\qquad+\left.\left(\int_0^T\|\f(t)\|_{\V'}\d t\right)^{1/2}+\|\D\|_{\mathcal{L}(\mathbb{U};\V')}\left(\int_0^T\|\U_{n_k}(t)\|_{\mathbb{U}}^2\d t\right)^{1/2}\right\}\left(\int_0^T\|\v(t)\|_{\V}^2\d t\right)^{1/2}. 
	\end{align*}}
	For $r=3$ and $\v\in\mathrm{L}^4(0,T;\wi\L^4),$ we estimate 
		{
	\begin{align*}
	&	\int_0^T|\langle\mathcal{C}(\u_n(t)),\v(t)\rangle|\d t\leq \int_0^T\|\u_n(t)\|_{\wi\L^4}^3\|\v(t)\|_{\wi\L^4}\d t\nonumber\\&\leq\left(\int_0^T\|\u_n(t)\|_{\wi\L^4}^4\d t\right)^{3/4}\left(\int_0^T\|\v(t)\|_{\wi\L^4}^4\d t\right)^{1/4},
	\end{align*}}
	so that $\mathcal{C}(\u_n)\in\mathrm{L}^{\frac{4}{3}}(0,T;\wi\L^{\frac{4}{3}})$. Following similar calculations as in the case of $r=2$, we obtain $\partial_t\u_n\in\mathrm{L}^{\frac{4}{3}}(0,T;\V')$ in the case of $r=3$.
	Hence, by using the Banach-Alaglou theorem, we can extract a subsequence $\{(\u_{n_k}, \U_{n_k})\}$ such that 
	\begin{equation}\label{conv}
	\left\{
	\begin{aligned} 
	\u_{n_k}&\xrightharpoonup{w^*}\u^*\ \text{ in } \ \mathrm{L}^{\infty}(0,T;\H), \ \u_{n_k} \rightharpoonup \u^*\ \text{ in  }\ \textrm{L}^2(0,T;\V),\\ 
\u_{n_k} &\rightharpoonup \u^*\ \text{ in  }\ \textrm{L}^{r+1}(0,T;\wi\L^{r+1}),\ \U_{n_k} \rightharpoonup  \U^*\ \text{ in  }\ \mathrm{L}^2(0,T;\mathbb{U}),\\
\partial_t\u_{n_k}&\rightharpoonup\partial_t\u\ \text{ in }\ \mathrm{L}^p(0,T;\V'),
	\end{aligned}
	\right.
	\end{equation}
as $k\to\infty$, where $p=2$ for $r=2$ and $p=3$ for $r=3$. 	Using the Aubin-Lion compactness Theorem   and the convergence in \eqref{conv}, we infer that 
	\begin{equation}
	\begin{aligned}
	\u_{n_k}&\to \u^*\text{ in }\mathrm{L}^{p}(0,T;\H), 
	\end{aligned}
	\end{equation}
	as $k\to\infty$ and along a subsequence, we obtain $	\u_{n_k}\to \u^*$ for a.e. $(t,x)\in[0,T]\times\Omega$. From the above convergences, one can pass limit in the equation corresponding to $(\u_{n_k},\U_{n_k})\in \mathscr{A}_{\mathrm{ad}}$ and conclude that  $\u^*(\cdot)$ is the unique weak solution of  \eqref{2.18} with the control $\U^*\in\mathrm{L}^2(0,T;\mathbb{U})$. For $r=2$, it is a consequence of the Theorem 3, section 5.9.2, \cite{evans} that $\u^* \in \mathrm{C}([0,T];\H) $. For the case of $r=3$, one can follow the same arguments as in \cite{GGP} to obtain $\u^* \in \mathrm{C}([0,T];\H) $.  It should be noted that the initial condition $
	\u_n(0)=	\u_0 \in\H,
	$ and the right continuity in time at $0$ gives 
$
	\u^*(0)=	\u_0\in\H.
$
	Furthermore, $\u^*(\cdot)$ satisfies the following energy equality:
	\begin{align}
&	\|\u^*(t)\|^2_{\H}+2\mu \int_0^t\|\u^*(s)\|^2_{\V}\d s+2\alpha\int_0^t\|\u^*(s)\|_{\H}^2\d s+2\beta\int_0^t\|\u^*(s)\|_{\widetilde{\L}^{r+1}}^{r+1}\d s\nonumber\\&=\|\u_0\|_{\H}^2+2\int_0^t\langle\f(s),\u^*(s)\rangle\d s+\int_0^t\langle\D\U^*(s),\u^*(s)\rangle\d s,
	\end{align}
	for all $t\in[0,T]$. 
	Since, $\u^*$ is the unique weak solution of  \eqref{2.18} with the control $\U^*\in\mathrm{L}^2(0,T;\mathbb{U})$, the whole sequence $\u_n$ converges to $\u^*$. This easily gives  $(\u^*,\U^*)\in \mathscr{A}_{\mathrm{ad}}$. 
	
	Finally, we show that  $(\u^*,\U^*)$ is a minimizer, that is \emph{$\mathscr{J}=\mathscr{J}(\u^*,\U^*)$}.  Since the cost functional $\mathscr{J}(\cdot,\cdot)$ is continuous and convex (see Proposition III.1.6 and III.1.10,  \cite{EI})  on $\mathrm{L}^2(0,T;\H)  \times \mathscr{U}_{\mathrm{ad}}$, it follows that $\mathscr{J}(\cdot,\cdot)$ is weakly lower semi-continuous (Proposition II.4.5, \cite{EI}). That is, for a sequence 
	$(\u_n,\U_n)\xrightharpoonup{w}(\u^*,\U^*)\ \text{ in }\ \mathrm{L}^2(0,T;\H) \times \mathrm{L}^2(0,T;\mathbb{U}),$
	we have 
$
	\mathscr{J}(\u^*,\U^*) \leq  \liminf \limits _{n\rightarrow \infty} \mathscr{J}(\u_n,\U_n).
$
	Therefore, we obtain 
	\begin{align*}\mathscr{J} \leq \mathscr{J}(\u^*,\U^*) \leq  \liminf \limits _{n\rightarrow \infty} \mathscr{J}(\u_n,\U_n)=  \lim \limits _{n\rightarrow \infty} \mathscr{J}(\u_n,\U_n) = \mathscr{J},\end{align*}
	and hence $(\u^*,\U^*)$ is a optimizer of the problem \eqref{control problem}.
\end{proof}

\section{First Order Necessary Conditions of Optimality}\label{sec4}\setcounter{equation}{0}
In this section, we prove the first order necessary condition of optimality for the optimal control problem \eqref{control problem}. We characterize optimal control in terms of the adjoint variable. In this section, we assume the following hypothesis on $h(\cdot)$. The function $h:\mathbb{U}\to\R$ is convex, lower semi-continuous and there exist positive constants $\varkappa_1$ and $\varkappa_2$ such that: \begin{align}\label{3.1}h(\U)\leq \varkappa_1\|\U\|_{\mathbb{U}}^2+\varkappa_2, \ \text{ for all } \ \u\in\mathbb{U}. \end{align}

We first provide Pontryagin's  maximum principle in terms of the \emph{Hamiltonian formulation}. Here we are providing  formal calculations only. Let us define the \emph{Lagrangian} by
$$\mathscr{L}(\u,\U) :=\frac{1}{2}\|\u-\u_d\|_{\H}^2 +\frac{1}{2}\|\nabla\u\|_{\H}^2+ h(\U), $$ so that $\mathscr{J}(\u,\U)=\int_0^T\mathscr{L}(\u(t),\U(t))\d t+\frac{1}{2}\|\u(T)-\u^f\|_{\H}^2$. 
We define the corresponding \emph{Hamiltonian} by
$$\mathscr{H}(\u,\U,\p):= \mathscr{L}(\u,\U) + \langle\p, \mathrm{N}(\u,\U) \rangle,$$ where $\mathrm{N}(\u,\U)=-\mu\A\u-\B(\u)-\alpha\u-\beta\mathcal{C}(\u)+\f+\U$ and the adjoint variable $\p(\cdot)$ satisfies system \eqref{adjp} given below. Since $(\u^*,\U^*)$ is an optimal pair for the problem \eqref{control problem}, Pontryagin's minimum principle  gives
\begin{eqnarray}
\mathscr{H}(\u^*(t),\U^*(t),\p(t)) \leq \mathscr{H}(\u^*(t),\mathrm{W},\p(t)),
\end{eqnarray}
for all $\mathrm{W} \in \mathbb{U}$ and a.e. $t\in[0,T]$. That is, the following minimum principle is satisfied by an optimal pair $(\u^*,\U^*)\in\mathscr{A}_{\mathrm{ad}}$ obtained in Theorem \ref{optimal}:
\begin{equation}\label{pm}
h(\U^*(t))+(\p(t),\D\U^*(t) ) \leq h(\mathrm{W})+( \p(t),\D\mathrm{W} ),
\end{equation} 
for all  $\mathrm{W}\in \mathbb{U},$ and a.e. $t\in[0,T]$. For $\mathscr{U}_{\mathrm{ad}}=\mathrm{L}^2(0,T;\mathbb{U})$, from \eqref{pm}, we see that $-\D^*\p \in \partial h(\U^*(t))\in \V', \text{ a.e. } t \in [0,T]$, where $\partial h(\cdot)$ denotes the subdifferential of $h(\cdot)$.
Since Pontryagin's maximum principle provides the first order necessary conditions of optimality, we show the final conclusions in the next theorem.  We follow similar techniques as in the works \cite{FART,BDM,MTM5,sritharan}, etc to get our main result.   The main Theorem of this section is:
\begin{theorem}\label{main}
	Let $(\u^*,\U^*)\in\mathscr{A}_{\mathrm{ad}}$ be an optimal solution of the problem \eqref{control problem} obtained in  Theorem \ref{optimal}. Then, there exists \emph{a unique weak solution} $\p\in\mathrm{C}([0,T];\H)\cap\mathrm{L}^2(0,T;\V)$ of the adjoint system:
	\begin{eqnarray}\label{adjp}
	\left\{
	\begin{aligned}
	-\partial_t\p(t)&+\mu\A\p(t)+(\B'(\u^*(t)))^*\p (t)+\alpha\p(t)+\beta\mathcal{C}'(\u^*(t))\p(t)\\&=(\u^*(t)-\u_d(t))-\Delta\u^*(t),\ \text{ in }\ \V', \\ \p(T)&=\u^*(T)-\u^f,
	\end{aligned}
	\right.
	\end{eqnarray} 
a.e. $t\in[0,T]$,  	such that
\begin{equation}\label{3.41}
-\D^*\p(t) \in \partial h(\U^*(t)), \ \text{ a.e. }\  t \in [0,T]. 
\end{equation}
\end{theorem}

Before embarking to the proof, we prove two important Lemmas, which is useful in establishing the above Theorem. From \eqref{40}, we know that 
\begin{align}
&\sup_{t\in[0,T]}\|\p(t)\|_{\H}^2+\mu\int_0^T\|\p(t)\|_{\V}^2\d t\nonumber\\& \leq  \left(\|\u^*(T)-\u^f\|_{\H}^2+\frac{4}{\mu} \int_0^T\|\u^*(t)-\u_d(t)\|_{\V'}^2\d t+\frac{4}{\mu}\int_0^T\|\u^*(t)\|_{\V}^2\d t\right)e^{\frac{8}{\mu^2}K_T}=:M_T,
\end{align} 
where 
	\begin{align}\label{4p7}
K_T=\|\u_0\|_{\H}^2+\frac{2}{\mu}\int_0^T\|\f(t)\|_{\V'}^2\d t+\frac{2}{\mu}\|\D\|_{\mathcal{L}(\mathbb{U},\V')}^2\int_0^T\|\U^*(t)\|_{\mathbb{U}}^2\d t.
\end{align}
\begin{lemma}\label{lem3.7}
	Let $(\u^*,\U^*)\in\mathscr{A}_{\mathrm{ad}}$ be an optimal pair for the control problem \eqref{control problem} obtained in  Theorem \ref{optimal}. Let $\u_{\U^*+\tau\U}$ be the unique weak solution to the system \eqref{2.18} with the control $\U^*+\tau\U$, for sufficiently small $\tau$ and $\U\in\mathscr{U}_{\mathrm{ad}}$. Then, we have 
	\begin{align}\label{325a}
	&\sup_{t\in[0,T]}\|\u_{\U^*+\tau\U}(t)-\u_{\U^*}(t)\|_{\H}^2+\mu\int_0^T\|\u_{\U^*+\tau\U}(t)-\u_{\U^*}(t)\|_{\V}^2\d t\nonumber\\&\quad+2\alpha\int_0^T\|\u_{\U^*+\tau\U}(t)-\u_{\U^*}(t)\|_{\H}^2\d t+\frac{\beta}{2^{r-2}}\int_0^T\|\u_{\U^*+\tau\U}(t)-\u_{\U^*}(t)\|_{\wi\L^{r+1}}^{r+1}\d t\nonumber\\&\leq  \left\{\frac{2\tau^2}{\mu}\|\D\|_{\mathcal{L}(\mathbb{U},\V')}^2\int_0^T\|\U(t)\|_{\mathbb{U}}^2\d t\right\}e^{\frac{8K_T}{\mu^2}},
	\end{align}
	where $K_T$ id defined in \eqref{4p7}. 
\end{lemma}
\begin{proof}
	Since $\u_{\U^*+\tau \U}$ and $\u_{\U^*}$ are the unique weak solutions of the system \eqref{2.18} corresponding to the  controls $\U^*+\tau \U$ and $\U^*$, respectively, we know that $\widetilde{\u}=\u_{\U^*+\tau \U}-\u_{\U^*}$  satisfies the following system:
	\begin{equation}\label{11}
	\left\{
	\begin{aligned}
	\frac{\partial \widetilde{\mathbf{u}}(t)}{\partial t}+\mu\A \widetilde{\mathbf{u}}(t)+\B(\u_{\U^*+\tau \U}(t))-\B(\u_{\U^*}(t))&+\alpha\widetilde{\mathbf{u}}(t)\\+\beta(\mathcal{C}(\u_{\U^*+\tau \U}(t))-\mathcal{C}(\u_{\U^*}(t)))&=\tau \D\U(t),\ \text{ in }\ [0,T],\\
	\widetilde{\u}(0)&=\mathbf{0}.
	\end{aligned}
	\right.
	\end{equation}
	Taking the inner product with $\widetilde{\u}(\cdot)$ to the first equation in \eqref{11}, we obtain 
	\begin{align}\label{322}
&	\frac{1}{2}\frac{\d }{\d t}\|\widetilde{\u}(t)\|_{\H}^2+\mu\|\widetilde{\u}(t)\|_{\V}^2+\alpha\|\widetilde{\u}(t)\|_{\H}^2\nonumber\\&=-\langle\B(\u_{\U^*+\tau \U}(t))-\B(\u_{\U^*}(t)),\widetilde{\u}(t)\rangle-\beta\langle\mathcal{C}(\u_{\U^*+\tau \U}(t))-\mathcal{C}(\u_{\U^*}(t)),\widetilde{\u}(t)\rangle+\tau\langle\D\U(t),\widetilde{\u}(t)\rangle\nonumber\\&\leq\|\u_{\U^*}(t)\|_{\V}\|\widetilde{\u}(t)\|_{\wi\L^4}^2 -\frac{\beta}{2^{r-1}}\|\widetilde{\u}(t)\|_{\wi\L^{r+1}}^{r+1}+\tau\|\D\U(t)\|_{\V'}\|\wi\u(t)\|_{\V}\nonumber\\&\leq-\frac{\beta}{2^{r-1}}\|\widetilde{\u}(t)\|_{\wi\L^{r+1}}^{r+1}+\frac{\mu}{2}\|\wi\u(t)\|_{\V}^2+\frac{2}{\mu}\|\u_{\U^*}(t)\|_{\V}^2\|\wi\u(t)\|_{\H}^2+\frac{\tau^2}{\mu}\|\D\|_{\mathcal{L}(\mathbb{U},\V')}^2\|\U(t)\|_{\mathbb{U}}^2,
	\end{align}
where we used  \eqref{2p8} and \eqref{a215}.
	Integrating the above inequality from $0$ to $t$, we obtain 
	\begin{align}\label{324}
&\|\widetilde{\u}(t)\|_{\H}^2+\mu\int_0^t\|\widetilde{\u}(s)\|_{\V}^2\d s+2\alpha\int_0^t\|\widetilde{\u}(s)\|_{\H}^2\d s+\frac{\beta}{2^{r-2}}\int_0^t\|\widetilde{\u}(s)\|_{\wi\L^{r+1}}^{r+1}\d s\nonumber\\&\leq \frac{4}{\mu}\int_0^t\|\u_{\U^*}(s)\|_{\V}^2\|\wi\u(s)\|_{\H}^2\d s+\frac{2\tau^2}{\mu}\|\D\|_{\mathcal{L}(\mathbb{U},\V')}^2\int_0^t\|\U(s)\|_{\mathbb{U}}^2\d s,
	\end{align}
for all $t\in[0,T]$. 	An application of Gronwall's inequality in \eqref{324} yields 
	\begin{align}\label{325}
&\|\widetilde{\u}(t)\|_{\H}^2\leq \left\{\frac{2\tau^2}{\mu}\|\D\|_{\mathcal{L}(\mathbb{U},\V')}^2\int_0^T\|\U(t)\|_{\mathbb{U}}^2\d t\right\}\exp\left(\frac{4}{\mu}\int_0^T\|\u_{\U^*}(t)\|_{\V}^2\d t\right),
	\end{align}
	for all $t\in[0,T]$,  which completes the proof. 
\end{proof}

The following lemma gives the differentiability of the mapping $\U \mapsto \u_{\U}$  from $\mathscr{U}_{\mathrm{ad}}$ into   $\mathrm{C}([0,T];\H)\cap\mathrm{L}^{2}(0,T;\V)$.
\begin{lemma}\label{lem3.8}
	Let $\u_0\in\H$ and $\f\in\mathrm{L}^2(0,T;\V')$ be given. Then the mapping $\U^* \mapsto \u_{\U^*}$  from $\mathscr{U}_{\mathrm{ad}}$ into  $\mathrm{C}([0,T];\H)\cap\mathrm{L}^{2}(0,T;\V)$ is Gateaux differentiable. Moreover, we have 
\begin{equation}\label{lim}
\left\{
\begin{aligned}
\lim_{\tau\downarrow 0}\frac{\|\u_{\U^*+\tau \U} (t)- \u_{\U^*}(t)-\tau\mathfrak{w}(t)\|_{\H}}{\tau}&=0, \ \text{ for all }\ t\in[0,T],\\
\lim_{\tau\downarrow 0}\frac{\|\u_{\U^*+\tau \U} - \u_{\U^*}-\tau\mathfrak{w}\|_{\mathrm{L}^{2}(0,T;\V)}}{\tau}&=0,
\end{aligned}
\right.
\end{equation}
	where $\mathfrak{w}(\cdot)$ is the \emph{unique weak solution} of the linearized system:
\begin{equation}\label{4a}
\left\{
\begin{aligned}
\partial_t\mathfrak{w}(t)+\mu\A\mathfrak{w}(t)+\B'(\u_{\U^*}(t))\mathfrak{w}(t)+\alpha\mathfrak{w}(t)+\beta\mathcal{C}'(\u_{\U^*}(t))\mathfrak{w}(t)&=\D\U(t),\ \text{ in }\ \V',\\
\mathfrak{w}(0)&=\mathbf{0}, 
\end{aligned}
\right.
\end{equation}
 a.e. $t\in[0,T]$,  and $\u_{\U^*}(\cdot)$  and $\u_{\U^*+\tau\U}(\cdot)$ are the unique weak solutions of the controlled system \eqref{2.18} with  the controls $\U^*$ and $\U^*+\tau\U$, respectively.
\end{lemma}

\begin{proof}
	Let us define 
$
	\y:=\u_{\U^*+\tau\U}-\u_{\U^*}-\tau\mathfrak{w}.
$
	Then $	\y(\cdot)$ satisfies the following system:
	\begin{equation}\label{4b}
	\left\{
	\begin{aligned}
	\frac{\partial \y(t)}{\partial t}&+\mu\A\y(t)+\B'(\mathbf{u}_{\U^*}(t))\y(t)+\alpha\y(t)+\beta\mathcal{C}'(\mathbf{u}_{\U^*}(t))\y(t)\\&=-\left[\B(\u_{\U^*+\tau\U}(t))-\B(\u_{\U^*}(t))-\B'(\u_{\U^*}(t))(\u_{\U^*+\tau\U}(t)-\u_{\U^*}(t))\right] \\& \quad -\beta\left[\mathcal{C}(\u_{\U^*+\tau\U}(t))-\mathcal{C}(\u_{\U^*}(t))-\mathcal{C}'(\mathbf{u}_{\U^*}(t))(\u_{\U^*+\tau\U}(t)-\u_{\U^*}(t))\right],\ \text{ in }\ \V',\\
	\y(0)&=\mathbf{0},
	\end{aligned}
	\right.
	\end{equation}
a.e. $t\in[0,T]$.  	Remember that the  terms $$\B(\u_{\U^*+\tau\U}-\u_{\U^*},\u_{\U^*+\tau\U}-\u_{\U^*})=\B(\u_{\U^*+\tau\U})-\B(\u_{\U^*})-\B'(\u_{\U^*})(\u_{\U^*+\tau\U}-\u_{\U^*})$$ and $$\mathcal{C}(\u_{\U^*+\tau\U})-\mathcal{C}(\u_{\U^*})-\mathcal{C}'(\mathbf{u}_{\U^*})(\u_{\U^*+\tau\U}-\u_{\U^*})$$ belong to $\mathrm{L}^p(0,T;\V')$, where $p=2$, for $r=2$ and $p=4/3$, for $r=3$,  and hence the system \eqref{4b} has a unique weak solution with $	\y\in\mathrm{C}([0,T];\H)\cap\mathrm{L}^2(0,T;\V).$	Taking the inner product with $\y(\cdot)$ to the first equation in \eqref{4b}, we find 
\begin{align}\label{3a27}
&\frac{1}{2}\frac{\d}{\d t}\|\y(t)\|_{\H}^2+\mu\|\y(t)\|_{\V}^2+\alpha\|\y(t)\|_{\H}^2+\beta\||\u_{\U^*}(t)|^{\frac{r-1}{2}}\y(t)\|_{\H}^2\nonumber\\&=-\langle\B(\y(t),\u_{\U^*}(t)),\y(t)\rangle-\langle\B(\u_{\U^*+\tau\U}(t)-\u_{\U^*}(t),\u_{\U^*+\tau\U}(t)-\u_{\U^*}(t)),\y(t)\rangle\nonumber\\&\quad-\beta\langle\left[\mathcal{C}(\u_{\U^*+\tau\U}(t))-\mathcal{C}(\u_{\U^*}(t))-\mathcal{C}'(\mathbf{u}_{\U^*}(t))(\u_{\U^*+\tau\U}(t)-\u_{\U^*}(t))\right],\y(t)\rangle. 
\end{align}
Using H\"older's, Ladyzhenskaya'a and Young's inequalities, we estimate $-\langle\B(\y,\u_{\U^*}),\y\rangle$ as 
\begin{align}\label{3a28}
-\langle\B(\y,\u_{\U^*}),\y\rangle&\leq\|\u_{\U^*}\|_{\V}\|\y\|_{\wi\L^4}^2\leq\sqrt{2}\|\u_{\U^*}\|_{\V}\|\y\|_{\H}\|\y\|_{\V}\leq\frac{\mu}{4}\|\y\|_{\V}^2+\frac{2}{\mu}\|\u_{\U^*}\|_{\V}^2\|\y\|_{\H}^2. 
\end{align}
A similar calculation yields 
\begin{align}
-\langle\B(\u_{\U^*+\tau\U}-\u_{\U^*},\u_{\U^*+\tau\U}-\u_{\U^*}),\y(t)\rangle\leq\frac{\mu}{4}\|\y\|_{\V}^2+\frac{2}{\mu}\|\u_{\U^*+\tau\U}-\u_{\U^*}\|_{\wi\L^4}^4. 
\end{align}
In order to estimate the term $-\beta\langle\left[\mathcal{C}(\u_{\U^*+\tau\U})-\mathcal{C}(\u_{\U^*})-\mathcal{C}'(\mathbf{u}_{\U^*})(\u_{\U^*+\tau\U}-\u_{\U^*})\right],\y\rangle$, we consider the cases $r=2$ and $r=3$ separately. For $r=2$, an application of Taylor's formula (see Theorem 7.9.1, \cite{PGC}), H\"older's and Young's inequalities yield
\begin{align}\label{3a30}
&-\beta\langle\left[\mathcal{C}(\u_{\U^*+\tau\U})-\mathcal{C}(\u_{\U^*})-\mathcal{C}'(\mathbf{u}_{\U^*})(\u_{\U^*+\tau\U}-\u_{\U^*})\right],\y\rangle\nonumber\\&=-\beta \left<\int_0^1\mathcal{C}'(\theta\u_{\U^*+\tau\U}+(1-\theta)\u_{\U^*})\d\theta (\u_{\U^*+\tau\U}-\u_{\U^*})-\mathcal{C}'(\mathbf{u}_{\U^*})(\u_{\U^*+\tau\U}-\u_{\U^*}),\y\right>\nonumber\\&=-\beta\int_0^1\langle\left[|\theta\u_{\U^*+\tau\U}+(1-\theta)\u_{\U^*}|-|\mathbf{u}_{\U^*}|\right](\u_{\U^*+\tau\U}-\u_{\U^*}),\y\rangle\d\theta\nonumber\\&\quad-\beta\int_0^1\left<\left[\frac{\theta\u_{\U^*+\tau\U}+(1-\theta)\u_{\U^*}}{|\theta\u_{\U^*+\tau\U}+(1-\theta)\u_{\U^*}|}(\theta\u_{\U^*+\tau\U}+(1-\theta)\u_{\U^*})-\frac{\u_{\U^*}}{|\u_{\U^*}|}\u_{\U^*}\right]\cdot(\u_{\U^*+\tau\U}-\u_{\U^*}),\y\right>\d\theta\nonumber\\&\leq \beta\sup_{0<\theta<1}\langle|\theta(\u_{\U^*+\tau\U}-\u_{\U^*})||\u_{\U^*+\tau\U}-\u_{\U^*}|,|\y|\rangle\nonumber\\&\quad+3\beta\sup_{0<\theta<1}\langle|\theta\u_{\U^*+\tau\U}+(1-\theta)\u_{\U^*}-\u_{\U^*}||\u_{\U^*+\tau\U}-\u_{\U^*}|,|\y|\rangle\nonumber\\&\leq 4\beta\|\u_{\U^*+\tau\U}-\u_{\U^*}\|_{\wi\L^4}^2\|\y\|_{\H}\nonumber\\&\leq\frac{\mu}{4}\|\y\|_{\V}^2+\frac{16}{\mu\lambda_1}\|\u_{\U^*+\tau\U}-\u_{\U^*}\|_{\wi\L^4}^4.
\end{align}
Combining \eqref{3a28}-\eqref{3a30}, substituting it in \eqref{3a27}  and then integrating from $0$ to $t$, we obtain 
\begin{align}\label{3a31}
&\|\y(t)\|_{\H}^2+\frac{\mu}{2}\int_0^t\|\y(s)\|_{\V}^2\d s+2\alpha\int_0^t\|\y(s)\|_{\H}^2\d s+2\beta\int_0^t\||\u_{\U^*}(s)|^{\frac{r-1}{2}}\y(s)\|_{\H}^2\d s\nonumber\\&\leq \frac{4}{\mu}\int_0^t\|\u_{\U^*}(s)\|_{\V}^2\|\y(s)\|_{\H}^2\d s+\frac{2}{\mu}\left(2+\frac{16}{\lambda_1}\right)\int_0^t\|\u_{\U^*+\tau\U}(s)-\u_{\U^*}(s)\|_{\wi\L^4}^4\d s.
\end{align}
An application of Gronwall's inequality in \eqref{3a31} yields  
\begin{align}\label{3a32}
\|\y(t)\|_{\H}^2&\leq \left\{\frac{2}{\mu}\left(2+\frac{16}{\lambda_1}\right)\int_0^T\|\u_{\U^*+\tau\U}(t)-\u_{\U^*}(t)\|_{\wi\L^4}^4\d t\right\}\exp\left(\frac{4}{\mu}\int_0^T\|\u_{\U^*}(t)\|_{\V}^2\d t\right)\nonumber\\&\leq\frac{2}{\mu^2}\left(2+\frac{16}{\lambda_1}\right)\left\{\frac{2\tau^2}{\mu}\|\D\|_{\mathcal{L}(\mathbb{U},\V')}^2\int_0^T\|\U(t)\|_{\mathbb{U}}^2\d t\right\}^2e^{\frac{12K_T}{\mu^2}},
\end{align}
for all $t\in[0,T]$, where we used \eqref{325a}. Passing $\tau\to 0$ in \eqref{3a32} yields the required result \eqref{lim} for the case $r=2$. 

For $r=3$, we estimate the term $-\beta\langle\left[\mathcal{C}(\u_{\U^*+\tau\U})-\mathcal{C}(\u_{\U^*})-\mathcal{C}'(\mathbf{u}_{\U^*})(\u_{\U^*+\tau\U}-\u_{\U^*})\right],\y\rangle$ as follows. Once again an application of Taylor's formula gives (see Theorem 6.5, \cite{RFC})
\begin{align}\label{3b30}
-&\beta\langle\left[\mathcal{C}(\u_{\U^*+\tau\U})-\mathcal{C}(\u_{\U^*})-\mathcal{C}'(\mathbf{u}_{\U^*})(\u_{\U^*+\tau\U}-\u_{\U^*})\right],\y\rangle\nonumber\\&\leq \frac{\beta}{2}\sup_{0<\theta<1}\|\mathcal{C}''(\theta\u_{\U^*+\tau\U}+(1-\theta)\u_{\U^*})(\u_{\U^*+\tau\U}-\u_{\U^*})\otimes(\u_{\U^*+\tau\U}-\u_{\U^*})\|_{\wi\L^{4/3}}\|\y\|_{\wi\L^4}\nonumber\\&\leq 3\beta\sup_{0<\theta<1}\|\theta\u_{\U^*+\tau\U}+(1-\theta)\u_{\U^*}\|_{\wi\L^4}\|\u_{\U^*+\tau\U}-\u_{\U^*}\|_{\wi\L^4}^2\|\y\|_{\wi\L^4}\nonumber\\&\leq 2^{1/4}3\beta\left(\|\u_{\U^*+\tau\U}\|_{\wi\L^4}+\|\u_{\U^*}\|_{\wi\L^4}\right)\|\u_{\U^*+\tau\U}-\u_{\U^*}\|_{\wi\L^4}^2\|\y\|_{\H}^{1/2}\|\y\|_{\V}^{1/2}\nonumber\\&\leq\frac{\mu}{4}\|\y\|_{\V}^2+ \frac{6^{1/3}9\beta^{4/3}}{4\mu^{1/3}} \left(\|\u_{\U^*+\tau\U}\|_{\wi\L^4}+\|\u_{\U^*}\|_{\wi\L^4}\right)^{4/3}\|\u_{\U^*+\tau\U}-\u_{\U^*}\|_{\wi\L^4}^{8/3}\|\y\|_{\H}^{2/3}\nonumber\\&\leq \frac{\mu}{4}\|\y\|_{\V}^2+\left(\|\u_{\U^*+\tau\U}\|_{\wi\L^4}^4+\|\u_{\U^*}\|_{\wi\L^4}^4\right)\|\y\|_{\H}^2+\frac{9\beta^2}{\sqrt{2\mu}}\|\u_{\U^*+\tau\U}-\u_{\U^*}\|_{\wi\L^4}^4,
\end{align}
and a calculation similar to \eqref{3a32} yields the required result. 
\end{proof}

Let us now prove the main Theorem \ref{main} by  invoking the Lemmas \ref{lem3.7} and \ref{lem3.8}. 
\begin{proof}[Proof of Theorem \ref{main}]
	Let $(\u^*,\U^*)\in\mathscr{A}_{\mathrm{ad}}$ be an optimal pair for the problem \eqref{control problem} established in Theorem \ref{optimal}. Let $\mathcal{G}(\U)=\mathscr{J}(\u_{\U},{\U})$, where $({\u}_{\U},\U)$ is the unique weak solution to the controlled system \eqref{2.18} with the control $\U\in\mathscr{U}_{\mathrm{ad}}$.  Let $\tau$ be sufficiently small such that $\U^*+\tau\U\in\mathscr{U}_{\mathrm{ad}}$ and $(\u_{\U^*+\tau\U},{\U^*+\tau\U})\in\mathscr{A}_{\mathrm{ad}}$. Then, we have (cf. \cite{BDM,MTM5}, etc)
	\begin{align}\label{lam}
&	\mathcal{G}(\U^*+ \tau \U) - \mathcal{G}(\U^*)\nonumber \\&= \mathscr{J}(\u_{\U^* + \tau \U},{\U^* + \tau \U})- \mathscr{J}(\u_{\U^*},{\U^*})  \nonumber\\       
	&= \frac{1}{2} \int_0^T  \|\u_{\U^* + \tau \U}(t)-\u_d(t)\|^2_{\H}  -\frac{1}{2} \int_0^T  \|\u_{\U^*}(t) -\u_d(t)\|^2_{\H}\d t \nonumber\\&\quad +\frac{1}{2}\int_0^T\|\nabla\u_{\U^*+\tau\U}(t)\|_{\H}^2\d t -\frac{1}{2}\int_0^T\|\nabla\u_{\U^*}(t)\|_{\H}^2\d t\nonumber\\&\quad + \int_0^T  h({\U^*(t) + \tau \U(t)})\d t- \int_0^T  h( \U^*(t))\d t +\frac{1}{2} \|\u_{\U^* + \tau \U}(T)-\u^f\|^2_{\H}-\frac{1}{2}\|\u_{\U^*}(T)-\u^f\|_{\H}^2 
\nonumber	\\ &= \frac{1}{2}\int_0^T  \|\u_{\U^*+\tau \U} (t)- \u_{\U^*}(t)\|^2_{\H} \d t+ \int_0^T (\u_{\U^*+\tau \U} (t)- \u_{\U^*}(t),\u_{\U^*}(t) -\u_d(t))\d t\nonumber\\&\quad+\frac{1}{2}\int_0^T\|\u_{\U^*+\tau \U} (t)- \u_{\U^*}(t)\|^2_{\V} \d t-\int_0^T\langle\u_{\U^*+\tau \U} (t)- \u_{\U^*}(t),\Delta\u_{\U^*}(t)\rangle\d t\nonumber\\&\quad+ \int_0^T  [h({\U^*(t) + \tau \U(t)})-h( \U^*(t))]\d t\nonumber\\&\quad +\frac{1}{2}\|\u_{\U^*+\tau \U} (T)- \u_{\U^*}(T)\|^2_{\H}+(\u_{\U^* + \tau \U}(T)-\u_{\U^* }(T),\u_{\U^*}(T)-\u^f) .
	\end{align}
	Using the estimate  \eqref{325a} (see Lemma \ref{lem3.7}), we know that $\|\u_{\U^*+\tau \U}(t) - \u_{\U^*}(t)\|^2_{\H},$ for all $t\in[0,T]$ and $\|\u_{\U^*+\tau \U} - \u_{\U^*}\|^2_{\mathrm{L}^2(0,T;\V)}$  can be estimated by $C\tau^2 \|\U\|^2_{\mathrm{L}^{2}(0,T;\mathbb{U})}$. Thus dividing by $\tau$, and then sending $\tau \downarrow 0$, we easily have  $\|\u_{\U^*+\tau \U} - \u_{\U^*}\|_{\mathrm{L}^2(0,T;\H)} \rightarrow 0$, $\|\u_{\U^*+\tau \U} - \u_{\U^*}\|_{\mathrm{L}^2(0,T;\V)}$ and $\|\u_{\U^*+\tau \U} (T)- \u_{\U^*}(T)\|_{\H}\to 0$  as $\tau \downarrow 0$. We denote the directional derivative of $h(\cdot)$ at $\U^*$ in the direction $\U$ as 
	\begin{align}
	h'(\U^*,\U)=\lim_{\lambda\downarrow 0} \frac{h(\U^*+\tau\U)-h(\U^*)}{\tau}. 
	\end{align}
	We also denote the directional derivative of $\mathcal{G}$ at $\U^*$ in the direction of $\U$ by $\mathcal{G}'(\U^*,\U)$. Let $\mathfrak{w}(\cdot)$ satisfies the linearized system \eqref{4a}.  From the Lemma \ref{lem3.8}, we also have the convergence \eqref{lim}.	Dividing by $\tau$ and then taking $\tau \downarrow 0$ in (\ref{lam}), we obtain
	\begin{align}\label{329}
	0 &\leq \mathcal{G}'(\U^*(t),\U(t)) = \underset{\tau \downarrow 0}{\lim} \frac{\mathcal{G}(\U^*(t)+ \tau \U(t))- \mathcal{G}(\U^*(t))}{\tau} \nonumber\\
	&= \int_0^T  \langle\mathfrak{w}(t),\u_{\U^*}(t) -\u_d(t)-\Delta\u_{\U^*}(t)\rangle\d t +  \int_0^Th'(\U^*(t),\U(t)) \d t
+(\mathfrak{w}(T),\u_{\U^*}(T)-\u^f)	\nonumber\\&= \int_0^T  \langle \mathfrak{w}(t),-\partial_t\p(t)+\mu\A\p(t)+(\B'(\u_{\U^*}(t)))^*\p (t)+\alpha\p(t)+\beta\mathcal{C}'(\u_{\U^*}(t))\p(t)\rangle\d t \nonumber\\&\quad+ \int_0^Th'(\U^*(t),\U(t)) \d t\nonumber\\&= \int_0^T\langle\partial_t\mathfrak{w}(t)+\mu\A\mathfrak{w}(t)+\B'(\u_{\U^*}(t))\mathfrak{w}(t)+\alpha\mathfrak{w}(t)+\beta\mathcal{C}'(\u_{\U^*}(t))\mathfrak{w}(t),\p(t)\rangle\d t\nonumber\\&\quad+ \int_0^Th'(\U^*(t),\U(t)) \d t\nonumber\\&= \int_0^T( \D\U (t),\p(t))\d t  + \int_0^Th'(\U^*(t),\U(t)) \d t,
	\end{align}  
	where we performed an integration by parts. Thus, from \eqref{329}, we infer that  
	$$0 \leq \mathcal{G}'(\U^*(t), \U(t))= \int_0^T(\U (t),\D^*\p(t))_{\mathbb{U}}\d t  + \int_0^Th'(\U^*(t),\U(t)) \d t,$$  for all $\U\in\mathrm{L}^2(0,T;\mathbb{U})$. Hence, from the above relation, it is immediate that 
	\begin{equation}\label{3a37}
-\D^*\p(t) \in \partial h(\U^*(t)), \ \text{ a.e. }\  t \in [0,T],
	\end{equation}
	which completes the proof. 
\end{proof}	

\begin{remark}\label{rem4.4}
1. 	If we take $h(\U)=\frac{1}{2}\|\U\|_{\mathbb{U}}^2$, then the function is Gateaux differentiable and the subdifferential contains a single point and from the relation \eqref{3a37}, we find 
	\begin{align}
	\U^*(t)=-\D^*\p(t),  \ \text{ a.e. }\  t \in [0,T]. 
	\end{align}
	
	2. The difficulty for considering the case for $r\in(3,\infty)$ is that the estimates similar to \eqref{3a30} and \eqref{3b30} are not controllable for $r>3$.   Thus, we are not able to obtain the convergence \eqref{lim} in Lemma \ref{lem3.8}, for $r>3$. One may think of modifying the admissible class $\mathscr{A}_{\mathrm{ad}}$ to the class of strong solutions $\u\in\C([0,T];\V)\cap\mathrm{L}^2(0,T;\D(\A))$,  with $\u_0\in\V$, $\f\in\mathrm{L}^2(0,T;\H)$ and $\D\in\mathcal{L}(\mathbb{U};\H)$. To the best of author's knowledge,  such strong solutions exist, for $r\in[1,3]$ only. As discussed in \cite{KT2}, we know that in bounded domains, $\mathcal{P}(|\u|^{r-1}\u)$ need not be zero on the boundary, and $\mathcal{P}$ and $-\Delta$ are not necessarily commuting (for a counterexample, see Example 2.19, \cite{JCR4}). Thus handling the term $(\mathcal{C}(\u),\A\u)=(\mathcal{P}(|\u|^{r-1}\u),\A\u),$ for $r>3$ is a difficult. Moreover, $\Delta\u\big|_{\partial\mathcal{O}}\neq 0$ in general and the term with pressure will not disappear (see \cite{KT2}), while taking the inner product with $-\Delta\u$ to the first equation in \eqref{1}. Therefore, the equality (\cite{KWH})
	\begin{align}\label{3}
	&\int_{\mathcal{O}}(-\Delta\u(x))\cdot|\u(x)|^{r-1}\u(x)\d x\nonumber\\&=\int_{\mathcal{O}}|\nabla\u(x)|^2|\u(x)|^{r-1}\d x+4\left[\frac{r-1}{(r+1)^2}\right]\int_{\mathcal{O}}|\nabla|\u(x)|^{\frac{r+1}{2}}|^2\d x\nonumber\\&=\int_{\mathcal{O}}|\nabla\u(x)|^2|\u(x)|^{r-1}\d x+\frac{r-1}{4}\int_{\mathcal{O}}|\u(x)|^{r-3}|\nabla|\u(x)|^2|^2\d x,
	\end{align}
	may not be useful in the context of bounded domains.
\end{remark}

\subsection{Uniqueness of optimal control in small time interval} In this subsection, we show the uniqueness of optimal control in small time interval for the optimal control problem \eqref{control problem} obtained in Theorem \ref{optimal}. Remember that the control to state mapping is nonlinear and getting  global in time unique optimal control is difficult. Thus, we are looking for a time $T$ such that this time ensures the uniqueness of optimal control. If we choose the final time $T$ to be  sufficiently small, then the state equation for $\u(\cdot)$ differ from the corresponding linearized problem $\mathfrak{w}(\cdot)$ slightly only.  In this case, the linearized state equation  produces a strictly convex cost functional and the corresponding optimal control is unique. Here, we take  $h(\U)=\frac{1}{2}\|\U\|_{\mathbb{U}}^2$ in \eqref{cost}.  
\begin{theorem}\label{thm4.5}
	Let $(\u^*,\U^*)\in\mathscr{A}_{\mathrm{ad}}$ be an  \emph{optimal pair} for the problem \eqref{control problem} with $h(\U)=\frac{1}{2}\|\U\|_{\mathbb{U}}^2$. Then, if the final time $T$ is sufficiently small, then an optimal pair $(\u^*,\U^*)\in\mathscr{A}_{\mathrm{ad}}$ obtained in the Theorem \ref{optimal}  is unique. 
\end{theorem}
\begin{proof}
	We assume that there exists an another optimal pair $(\widetilde{\u},\widetilde{\U})\in\mathscr{A}_{\mathrm{ad}}$. 
From the Theorem \ref{main}, 	we know that  $\U^*(t)=-\D^*\p^*(t)\in\mathbb{U}$, a.e. $t\in[0,T]$, and $\widetilde{\U}(t)=-\D^*\widetilde{\p}(t)\in\mathbb{U}$, a.e. $t\in[0,T]$. Note also that $\p^*(\cdot)$ satisfies the adjoint system \eqref{adjp}  and $\widetilde{\p}(\cdot)$ satisfies the adjoint system \eqref{adjp} with $\u^*(\cdot)$ replaced by $\wi\u(\cdot)$, and  the forcing $(\u^*(\cdot)-\u_d(\cdot))-\Delta\u^*(\cdot)$ replaced by $(\widetilde{\u}(\cdot)-\u_d(\cdot))-\Delta\wi\u(\cdot)$. Thus, we have 
	\begin{align}\label{3.46}
\int_0^T	\|\D(\U^*(t)-\widetilde{\U}(t))\|_{\mathbb{V}'}^2\d t=\int_0^T\|\p^*(t)-\widetilde{\p}(t)\|_{\mathbb{V}'}^2\d t\leq\frac{1}{\lambda_1^2}\int_0^T\|\p^*(t)-\widetilde{\p}(t)\|_{\mathbb{V}}^2\d t. 
	\end{align}
	The quantity on the right hand side of the inequality \eqref{3.46} can be estimated similarly as in  \eqref{40}, as the system satisfied by $\p(\cdot)=\widetilde{\p}(\cdot)-\p^*(\cdot)$ is given by 
	\begin{eqnarray}\label{3.47}
	\left\{
	\begin{aligned}
	-\partial_t\p(t)+\mu{\A}\p(t)+(\B'(\widetilde{\u}(t)))^*\wi\p(t)-(\B'(\u^*(t)))^*\p^* (t)&+\alpha\p(t)\\+\beta(\mathcal{C}'(\widetilde{\u}(t))\wi\p(t)-\mathcal{C}'(\u^*(t))\p^*(t))&= \u(t)-\Delta\u(t), \ \text{ in } \ \V',\\ \p(T)&=\u(T), 
	\end{aligned}
	\right.
	\end{eqnarray}
	for a.e. $t\in[0,T]$,  where $\u=\wi\u-\u^*$.  Taking the inner product with $\p(\cdot)$ to the first equation in \eqref{3.47} and then integrating from $t$ to $T$, we find 
	\begin{align}\label{3.49}
	&\|\p(t)\|_{\H}^2+2\mu\int_t^T\|\p(s)\|_{\V}^2\d s+2\alpha\int_t^T\|\p(s)\|_{\H}^2\d s\nonumber\\&=\|\u(T)\|_{\H}^2 -2\int_t^T\langle(\B'(\widetilde{\u}(s)))^*\wi\p(s)-(\B'(\u^*(s)))^*\p^* (s),\p(s)\rangle\d s\nonumber\\&\quad-2\beta\int_t^T\langle\mathcal{C}'(\widetilde{\u}(s))\wi\p(s)-\mathcal{C}'(\u^*(s))\p^*(s),\p(s)\rangle\d s +2\int_t^T\langle\u(s)-\Delta\u(s),\p(s)\rangle\d s,
	\end{align}
for all $t\in[0,T]$. Note that 
\begin{align}\label{3a43}
\langle(\B'(\widetilde{\u}))^*\wi\p-(\B'(\u^*))^*\p^* ,\p\rangle&=\langle (\B'(\widetilde{\u}))^*\p,\p\rangle +\langle((\B'(\widetilde{\u}))^*-(\B'(\u^*))^*)\p^*,\p\rangle\nonumber\\&=\langle\B(\p,\wi\u),\p\rangle+\langle\B(\u,\p),\p^*\rangle+\langle\B(\p,\u),\p^*\rangle.
\end{align}
Using H\"older's, Ladyzhenskaya's and Young's inequalities, we estimate the terms in the right hand side of the inequality \eqref{3a43} as 
\begin{align}\label{3a44}
|\langle\B(\p,\wi\u),\p\rangle|&\leq\|\wi\u\|_{\V}\|\p\|_{\wi\L^4}^2\leq\sqrt{2}\|\wi\u\|_{\V}\|\p\|_{\H}\|\p\|_{\V}\leq\frac{\mu}{8}\|\p\|_{\V}^2+\frac{4}{\mu}\|\wi\u\|_{\V}^2\|\p\|_{\H}^2,\\
|\langle\B(\u,\p),\p^*\rangle|&\leq\|\u\|_{\wi\L^4}\|\p\|_{\V}\|\p^*\|_{\wi\L^4}\leq\frac{\mu}{8}\|\p\|_{\V}^2+\frac{2}{\mu}\|\u\|_{\wi\L^4}^2\|\p^*\|_{\wi\L^4}^2, \\
|\langle\B(\p,\u),\p^*\rangle|&\leq\|\p\|_{\wi\L^4}\|\u\|_{\V}\|\p^*\|_{\wi\L^4}\leq 2^{1/4}\|\p\|_{\H}^{1/2}\|\p\|_{\V}^{1/2}\|\u\|_{\V}\|\p^*\|_{\wi\L^4}\nonumber\\&\leq\frac{\mu}{8}\|\p\|_{\V}^2+\frac{3}{4}\left(\frac{2^{2/3}}{\mu^{1/3}}\right)\|\u\|_{\V}^{4/3}\|\p^*\|_{\wi\L^4}^{4/3}\|\p\|_{\H}^{2/3}\nonumber\\&\leq \frac{\mu}{8}\|\p\|_{\V}^2+\frac{1}{2}\|\p^*\|_{\wi\L^4}^4\|\p\|_{\H}^2+\frac{1}{\sqrt{2\mu}}\|\u\|_{\V}^2. 
\end{align}
For $r=2$, the term $-2\beta\langle\mathcal{C}'(\widetilde{\u})\wi\p-\mathcal{C}'(\u^*)\p^*,\p\rangle$ can be estimated as 
\begin{align}\label{3a47}
-2\beta\langle\mathcal{C}'(\widetilde{\u})\wi\p-\mathcal{C}'(\u^*)\p^*,\p\rangle&=-2\beta\langle\mathcal{C}'(\widetilde{\u})\p,\p\rangle-2\beta\langle(\mathcal{C}'(\widetilde{\u})-\mathcal{C}'(\u^*))\p^*,\p\rangle\nonumber\\&\leq -2\beta\langle(|\widetilde{\u}|-|\u^*|)\p^*,\p\rangle-2\beta\left<\frac{\wi\u}{|\wi\u|}(\wi\u\cdot\p^*)-\frac{\u^*}{|\u^*|}(\u^*\cdot\p^*),\p\right>\nonumber\\&\leq 8\beta\|\u\|_{\wi\L^4}\|\p^*\|_{\wi\L^4}\|\p\|_{\H}\nonumber\\&\leq\frac{\mu}{2}\|\p\|_{\V}^2+\frac{32\beta^2}{\lambda_1\mu}\|\u\|_{\wi\L^4}^2\|\p^*\|_{\wi\L^4}^2, 
\end{align} 
since $\langle\mathcal{C}'(\widetilde{\u})\p,\p\rangle\geq 0$. 
Similarly, we obtain 
\begin{align}\label{3a48}
|\langle\u-\Delta\u,\p\rangle| \leq\|\u\|_{\H}\|\p\|_{\H}+\|\u\|_{\V}\|\p\|_{\V}\leq\frac{\mu}{8}\|\p\|_{\V}^2+\frac{4}{\mu}\left(1+\frac{1}{\lambda_1^2}\right)\|\u\|_{\V}^2. 
\end{align}
Combining \eqref{3a44}-\eqref{3a47} and substituting it in \eqref{3.49}, we get 
\begin{align}\label{3a49}
	&\|\p(t)\|_{\H}^2+\frac{\mu}{2}\int_t^T\|\p(s)\|_{\V}^2\d s+2\alpha\int_t^T\|\p(s)\|_{\H}^2\d s\nonumber\\&\leq\|\u(T)\|_{\H}^2+\frac{8}{\mu}\int_t^T\|\wi\u(s)\|_{\V}^2\|\p(t)\|_{\H}^2\d s+\frac{4}{\mu}\left(1+\frac{16\beta^2}{\lambda_1\mu}\right)\int_t^T\|\u(s)\|_{\wi\L^4}^2\|\p^*(s)\|_{\wi\L^4}^2\d s\nonumber\\&\quad+\int_t^T\|\p^*(s)\|_{\wi\L^4}^4\|\p(s)\|_{\H}^2\d s+\left[\sqrt{\frac{2}{\mu}}+\frac{8}{\mu}\left(1+\frac{1}{\lambda_1^2}\right)\right]\int_t^T\|\u(s)\|_{\V}^2\d s .
\end{align}
	An application of Gronwall's inequality in \eqref{3a49} yields 
	\begin{align}\label{3a50} 
	\|\p(t)\|_{\H}^2&\leq\bigg\{\|\u(T)\|_{\H}^2+\frac{4}{\mu}\left(1+\frac{16\beta^2}{\lambda_1}\right)\left(\int_0^T\|\u(t)\|_{\wi\L^4}^4\d t\right)^{1/2}\left(\int_0^T\|\p^*(t)\|_{\wi\L^4}^4\d t\right)^{1/2}\nonumber\\&\qquad+\left[\sqrt{\frac{2}{\mu}}+\frac{8}{\mu}\left(1+\frac{1}{\lambda_1^2}\right)\right]\int_0^T\|\u(t)\|_{\V}^2\d t \bigg\}\nonumber\\&\quad\times\exp\left\{\frac{8}{\mu}\int_0^T\|\wi\u(t)\|_{\V}^2\d t+\int_0^T\|\p^*(t)\|_{\wi\L^4}^4\d t\right\},
	\end{align}
	for all $t\in[0,T]$. Using \eqref{3a50} in \eqref{3a49} and then taking $t=0$, we find 
	\begin{align}\label{3.51}
&\int_0^T\|\p(t)\|_{\V}^2\d t\nonumber\\&\leq\frac{2}{\mu} \bigg\{\|\u(T)\|_{\H}^2+\frac{4}{\mu}\left(1+\frac{16\beta^2}{\lambda_1}\right)\left(\int_0^T\|\u(t)\|_{\wi\L^4}^4\d t\right)^{1/2}\left(\int_0^T\|\p^*(t)\|_{\wi\L^4}^4\d t\right)^{1/2}\nonumber\\&\quad +\left[\sqrt{\frac{2}{\mu}}+\frac{8}{\mu}\left(1+\frac{1}{\lambda_1^2}\right)\right]\int_0^T\|\u(t)\|_{\V}^2\d t \bigg\}\exp\left\{\frac{16}{\mu}\int_0^T\|\wi\u(t)\|_{\V}^2\d t+2\int_0^T\|\p^*(t)\|_{\wi\L^4}^4\d t\right\}\nonumber\\&\leq \left\{\frac{2}{\mu}\int_0^T\|\D(\U^*(t)-\wi\U(t))\|_{\mathbb{V}'}^2\d t\right\}e^{\frac{2}{\mu}\left(\frac{2K_T}{\mu}+M_T\right)} \nonumber\\&\quad\times\frac{2}{\mu} \bigg\{1+\frac{4}{\mu^2}\left(1+\frac{16\beta^2}{\lambda_1}\right)M_T+\left[\sqrt{\frac{2}{\mu}}+\frac{8}{\mu}\left(1+\frac{1}{\lambda_1^2}\right)\right]\bigg\}\nonumber\\&\quad\times\exp\left\{\frac{16}{\mu^3}\int_0^T\|\D(\U^*(t)-\wi\U(t))\|_{\mathbb{V}'}^2\d t\right\}\exp\left(e^{\frac{8K_T}{\mu^2}}\right),
	\end{align}
	where we used \eqref{325a} and \eqref{40}. Combining \eqref{3.46} and \eqref{3.51}, it can be easily seen that 
	\begin{align}
\int_0^T	\|\D(\U^*(t)-\widetilde{\U}(t))\|_{\mathbb{V}'}^2\d t&\leq \left\{\frac{4}{\lambda_1^2\mu^2}\int_0^T\|\D(\U^*(t)-\wi\U(t))\|_{\mathbb{V}'}^2\d t\right\}e^{\frac{2}{\mu}\left(\frac{2K_T}{\mu}+M_T\right)}\nonumber\\&\quad\times  \bigg\{1+\frac{4}{\mu^2}\left(1+\frac{16\beta^2}{\lambda_1}\right)M_T+\left[\sqrt{\frac{2}{\mu}}+\frac{8}{\mu}\left(1+\frac{1}{\lambda_1^2}\right)\right]\bigg\}\nonumber\\&\quad\times\exp\left\{\frac{16}{\mu^3}\int_0^T\|\D(\U^*(t)-\wi\U(t))\|_{\mathbb{V}'}^2\d t\right\}\exp\left(e^{\frac{8K_T}{\mu^2}}\right).
	\end{align}
	From the above relation, one can choose a time $T$ sufficiently small such that $$\int_0^T	\|\D(\U^*(t)-\widetilde{\U}(t))\|_{\mathbb{V}'}^2\d t\leq 0.$$	Thus, we obtain $\U^*(t)=\widetilde{\U}(t)$, a. e. $t\in[0,T]$, for sufficiently small $T$.  
	
	For the case $r=3$, we need to estimate the term $-2\beta\langle\mathcal{C}'(\widetilde{\u})\wi\p-\mathcal{C}'(\u^*)\p^*,\p\rangle$  only. An application of Taylor's formula yields 
	\begin{align}\label{3a53}
	-&2\beta\langle\mathcal{C}'(\widetilde{\u})\wi\p-\mathcal{C}'(\u^*)\p^*,\p\rangle\nonumber\\&\leq -2\beta\langle(\mathcal{C}'(\widetilde{\u})-\mathcal{C}'(\u^*))\p^*,\p\rangle\nonumber\\&\leq\beta\sup_{0<\theta<1}\|\mathcal{C}''(\theta\wi\u+(1-\theta)\u^*)(\u\otimes\p^*)\|_{\wi\L^{\frac{4}{3}}}\|\p\|_{\wi\L^4}\nonumber\\&\leq 2^{1/4}6\beta(\|\wi\u\|_{\wi\L^4}+\|\wi\u^*\|_{\wi\L^4})\|\u\|_{\wi\L^4}\|\p^*\|_{\wi\L^4}\|\p\|_{\H}^{1/2}\|\p\|_{\V}^{1/2}\nonumber\\&\leq\frac{\mu}{2}\|\p\|_{\V}^2+\frac{3}{4\mu^{1/3}}(6\beta)^{4/3}(\|\u\|_{\wi\L^4}+2\|\wi\u^*\|_{\wi\L^4})^{4/3}\|\u\|_{\wi\L^4}^{4/3}\|\p^*\|_{\wi\L^4}^{4/3}\|\p\|_{\H}^{2/3}\nonumber\\&\leq \frac{\mu}{2}\|\p\|_{\V}^2+\frac{1}{2}\|\p^*\|_{\wi\L^4}^4\|\p\|_{\H}^2+\frac{18\sqrt{2}\beta^2}{\sqrt{\mu}}\|\u\|_{\wi\L^4}^4+\frac{72\sqrt{2}\beta^2}{\sqrt{\mu}}\|\u^*\|_{\wi\L^4}^2\|\u\|_{\wi\L^4}^2,
	\end{align} 
	and the proof can be completed using similar arguments as in the case of $r=2$. 
\end{proof}

\section{Second Order Necessary and Sufficient Conditions of Optimality}\label{sec5}\setcounter{equation}{0} 
In this section, we obtain the second order necessary and sufficient conditions of optimality for the problem \eqref{control problem}.  We prove it for the case $r=3$ only, due to the lack of second order Gateaux  derivative of $\mathcal{C}(\u)$ for the case of $r=2$.  Let $(\widehat{\u},\widehat{\U})\in\mathcal{A}_{\mathrm{ad}}$ be an arbitrary feasible pair for the optimal control problem \eqref{control problem}. Let us set 
\begin{align}\label{def Q}
\mathcal{M}_{\widehat{\u},\widehat{\U}}:= \{(\u,\U) \in \mathcal{A}_{\mathrm{ad}} \} - \{(\widehat{\u},\widehat{\U})\},
\end{align}
which denotes the differences of all feasible pairs of the problem \eqref{control problem} with  $(\widehat{\u},\widehat{\U})$. The next two Theorems provide the second order necessary and sufficient optimality condition for the optimal control problem \eqref{control problem} for the case $r=3$. In this section also, we take  $h(\U)=\frac{1}{2}\|\U\|_{\mathbb{U}}^2$ in \eqref{cost}.   Motivated from \cite{CT}, similar problems for Navier-Stokes equations have been considered in \cite{TFWD,LijuanPezije},  Cahn-Hillard-Navier-Stokes equations have been addressed in \cite{BDM3} and  the primitive equations of the ocean have been given in \cite{TTM}.
\begin{theorem}[Necessary condition]\label{necessary}
	Let  $(\u^*,\U^*)$ be an optimal pair for the problem \eqref{control problem}. Then for any $(\u,\U) \in \mathcal{M}_{\u^*,\U^*}$, we have 
	\begin{align}\label{417}
	&\frac{1}{2}\left[\int_0^T\left(\|\u(t)\|_{\H}^2+\|\nabla\u(t)\|_{\H}^2+\|\U(t)\|_{\mathbb{U}}^2\right)\d t \right]+\frac{1}{2}\|\u(T)\|_{\H}^2-\int_0^T\langle\B(\u(t),\u(t)),\p(t)\rangle\d t\nonumber \\ &\quad-\beta\int_0^T\left<2(\u^*(t)\cdot\u(t))\u(t)+|\u(t)|^2(\u(t)+\u^*(t)),\p(t)\right>\d t\geq 0. 
	\end{align}
\end{theorem}
\begin{proof}
	For any $(\u,\U)\in \mathcal{M}_{\u^*,\U^*},$ by the definition given in \eqref{def Q}, there exists $(\z,\W) \in \mathcal{A}_{\mathrm{ad}}$ such that $(\u,\U)= (\z - \u^*,\W - \U^*)$. Thus, from \eqref{2.18}, we can derive that $(\u,\U)$ satisfies the following system:
	\begin{eqnarray}\label{111}
	\left\{
	\begin{aligned}
\partial_t\u(t)+\mu\A\mathbf{u}(t)+\B(\mathbf{u}(t)+\u^*(t))-\B(\u^*(t))&+\alpha\u(t)\\+\beta(\mathcal{C}(\u(t)+\u^*(t))-\mathcal{C}(\u^*(t)))&=\D\U(t),\ \text{ in }\ \V',\\
	\mathbf{u}(0)&=\mathbf{0}.
	\end{aligned}
	\right.
	\end{eqnarray}
	a.e. $t\in[0,T]$. 	Taking the inner product with $\p(\cdot)$ to the first equation in \eqref{111} and then integrating over $[0,T]$,  we obtain
	\begin{align}
	&\int_0^T \langle\partial_t{\u}(t) +\mu\A\mathbf{u}(t)+ \B'(\u^*(t))\u(t)+\alpha\u(t)+\beta\mathcal{C}'(\u^*(t))\u(t)-\D\U(t),\p (t)\rangle \d t \nonumber\\&\quad+\int_0^T \langle\B(\mathbf{u}(t)+\u^*(t))-\B(\u^*(t))-\B'(\u^*(t))\u(t),\p (t)\rangle \d t \nonumber\\&\quad+\beta \int_0^T (\mathcal{C}(\u(t)+\u^*(t))-\mathcal{C}(\u^*(t))-\mathcal{C}'(\u^*(t))\u(t),\p(t)) \d t =0 .
	\end{align}
	Using an integration by parts, we infer that 
	\begin{align}\label{3.18}
	&(\u(T),\p(T))+\int_0^T \langle\u(t), -{\p}_t(t) +\mu\A\mathbf{p}(t)+ (\B'(\u^*(t)))^*\p(t)+\alpha\p(t)+\beta\mathcal{C}'(\u^*(t))\p(t)\rangle  \d t\nonumber\\&\quad+\int_0^T  \langle\B(\mathbf{u}(t)+\u^*(t))-\B(\u^*(t))-\B'(\u^*(t))\u(t)-\D\U(t),\p (t)\rangle \d t \nonumber\\&\quad+\beta \int_0^T \langle\mathcal{C}(\u(t)+\u^*(t))-\mathcal{C}(\u^*(t))-\mathcal{C}'(\u^*(t))\u(t),\p(t)\rangle  \d t =0 .
	\end{align}
	Since $(\u^*,\U^*)$ is an optimal pair, we know that $(\u^*,\U^*)$  satisfies the first order necessary condition
	\begin{align}\label{410}
\U^*(t)=-\D^*\p(t), \ \text{ for a.e. }\ t\in[0,T]. 
	\end{align}
	Using  \eqref{410} and \eqref{adjp} in \eqref{3.18}, we find
	\begin{align}\label{5.32}
	&(\u(T),\u^*(T)-\u^f)+ \int_0^T ( \U^*(t),\U(t))_{\mathbb{U}}\d t +\int_0^T \langle\u^*(t)-\u_d(t)-\Delta\u^*(t),\u(t)\rangle \d t
\nonumber\\&\quad+\int_0^T  \langle\B(\mathbf{u}(t)+\u^*(t))-\B(\u^*(t))-\B'(\u^*(t))\u(t),\p (t)\rangle \d t \nonumber\\&\quad+\beta \int_0^T \langle\mathcal{C}(\u(t)+\u^*(t))-\mathcal{C}(\u^*(t))-\mathcal{C}'(\u^*(t))\u(t),\p(t)\rangle \d t =0 .
	\end{align}
	Since $(\u,\U) \in \mathcal{M}_{\u^*,\U^*}$, by the definition \eqref{def Q}, we know that  $(\u+\u^*,\U+\U^*)$ is a feasible pair for the problem \eqref{control problem}. Remembering that  $(\u^*,\U^*)$ is an optimal pair, from \eqref{cost}, we also obtain 
	\begin{align}\label{412}
	&\mathscr{J}(\u+\u^*,\U+\U^*) - \mathscr{J}(\u^*,\U^*) \nonumber\\&=\frac{1}{2} \int_0^T [ \|\u(t)+\u^*(t)-\u_d(t)\|_{\H}^2+\|\nabla(\u(t)+\u^*(t))\|_{\H}^2] \d t+\frac{1}{2}\int_0^T\|\U(t)+\U^*(t)\|_{\mathbb{U}}^2\d t\nonumber\\&\quad+\frac{1}{2}\|\u(T)+\u^*(T)-\u^f\|_{\H}^2  - \frac{1}{2} \int_0^T [ \|\u^*(t)-\u_d(t)\|_{\H}^2+\|\nabla\u^*(t)\|_{\H}^2] \d t-\frac{1}{2}\int_0^T\|\U^*(t)\|_{\mathbb{U}}^2\d t\nonumber\\&\quad-\frac{1}{2}\|\u^*(T)-\u^f\|_{\H}^2 \geq 0.
	\end{align} 
Thus, it is immediate that 
	\begin{align}
	&\int_0^T \left[( \u^*(t)-\u_d(t),\u(t) )+ \frac{1}{2}\|\u(t)\|_{\H}^2+\langle-\Delta\u^*(t),\u(t)\rangle+\frac{1}{2}\|\nabla\u(t)\|_{\H}^2\right] \d t   \nonumber \\	&\quad +\int_0^T \left[(\U^*(t),\U(t) )_{\mathbb{U}} + \frac{1}{2}\|\U(t)\|^2_{\mathbb{U}}\right] \d t+(\u^*(T)-\u^f,\u(T))+\frac{1}{2}\|\u(T)\|_{\H}^2 \geq 0.
	\end{align}
	From \eqref{5.32}, it follows that
	\begin{align}\label{415}
	&\frac{1}{2}\left[\int_0^T\left(\|\u(t)\|_{\H}^2+\|\nabla\u(t)\|_{\H}^2+\|\U(t)\|_{\mathbb{U}}^2\right)\d t \right]+\frac{1}{2}\|\u(T)\|_{\H}^2\nonumber \\ 
	& \quad-\int_0^T  \langle\B(\mathbf{u}(t)+\u^*(t))-\B(\u^*(t))-\B'(\u^*(t))\u(t),\p (t)\rangle \d t \nonumber\\&\quad-\beta \int_0^T \langle\mathcal{C}(\u(t)+\u^*(t))-\mathcal{C}(\u^*(t))-\mathcal{C}'(\u^*(t))\u(t),\p(t)\rangle \d t \geq 0. 
	\end{align}
	But, we know that 
	\begin{align}\label{511}
	 \langle\B(\mathbf{u}+\u^*)-\B(\u^*)-\B'(\u^*)\u,\p \rangle&=\langle\B(\u,\u)+\B(\u,\u^*)+\B(\u^*,\u)+\B(\u^*,\u^*),\p\rangle\nonumber\\&\quad-\langle\B(\u^*,\u^*),\p\rangle -\langle\B(\u^*,\u)-\B(\u,\u^*),\p\rangle \nonumber\\&=\langle\B(\u,\u),\p\rangle. 
	\end{align}
	Using Taylor's series expansion (see Theorem 7.9.1, \cite{PGC}), we have (see Remark \ref{rem2.2})
	\begin{align}\label{416}
&\langle\mathcal{C}(\u+\u^*)-\mathcal{C}(\u^*)-\mathcal{C}'(\u^*)\u,\p\rangle\nonumber\\&= \frac{1}{2}\left<\mathcal{C}''(\u^*)(\u\otimes\u),\p\right>+\frac{1}{2}\left<\int_0^1(1-\theta)^{2}(\mathcal{C}'''(\u^*+\theta\u)(\u\otimes\u\otimes\u))\d\theta,\p\right>\nonumber\\&=2\langle(\u^*\cdot\u)\u,\p\rangle +\langle|\u|^2\u^*,\p\rangle+\langle|\u|^2\u,\p\rangle.
	\end{align}
	Using \eqref{511} and \eqref{416} in \eqref{415}, we finally get \eqref{417}.
\end{proof}

\begin{remark}
For $r=1$, the condition \eqref{417} becomes 
		\begin{align}
	&\frac{1}{2}\left[\int_0^T\left(\|\u(t)\|_{\H}^2+\|\nabla\u(t)\|_{\H}^2+\|\U(t)\|_{\mathbb{U}}^2\right)\d t \right]+\frac{1}{2}\|\u(T)\|_{\H}^2-\int_0^T\langle\B(\u(t),\u(t)),\p(t)\rangle\d t\geq 0,
	\end{align}
	which is same as the case for the optimal control problems governed by the 2D Navier-Stokes equations (cf. \cite{LijuanPezije}). 
\end{remark}

\begin{theorem}[Sufficient condition]\label{sufficient}
	Let  $(\u^*,\U^*)$ be a feasible pair for the problem \eqref{control problem} and assume that the first order necessary condition holds (see  \eqref{410}). Let us also assume that for any $(\u,\U) \in \mathcal{M}_{\u^*,\U^*},$ the  inequality \eqref{417} holds. 
	Then $(\u^*,\U^*)$ is an optimal pair for the problem \eqref{control problem}.
\end{theorem}
\begin{proof}
	For any $(\z,\W) \in \mathcal{A}_{\mathrm{ad}}$, by the definition of \eqref{def Q}, we know that $(\z - \u^*,\W - \U^*) \in \mathcal{M}_{\u^*,\U^*}$ and  it satisfies:
	\begin{eqnarray}\label{1111}
	\left\{
	\begin{aligned}
	\partial_t (\z (t)- \u^*(t))+\mu\A(\z(t) - \u^*(t))+\B(\z(t)) - \B(\u^*(t))&+\alpha(\z(t)-\u^*(t))\\ +\beta(\mathcal{C}(\z(t))-\mathcal{C}(\u^*(t)))&=\D(\W(t)-\U^*(t)),\ \text{ in }\ \V',\\
	\z(0) - \u^*(0)&=\mathbf{0}.
	\end{aligned}
	\right.
	\end{eqnarray}
	Taking the inner product with $\p(\cdot)$ to the first two equations in \eqref{1111}, integrating over $[0,T]$ and then performing an integration by parts, we get  
	\begin{align}\label{419}
	& \int_0^T \langle\z(t)-\u^*(t), -\partial_t\p(t) +\mu\A\mathbf{p}(t)+ \B'(\u^*(t))\p(t)+\alpha\p(t)+\beta\mathcal{C}'(\u^*(t))\p(t)\rangle \d t \nonumber\\&\quad +\int_0^T\langle\B(\z(t))- \B(\u^*(t))- \B'(\u^*(t))(\z(t)-\u^*(t))-(\W(t)-\U^*(t),\D^*\p(t))_{\mathbb{U}}\rangle\d t\nonumber\\&\quad+\beta\int_0^T\langle\mathcal{C}(\z(t))-\mathcal{C}(\u^*(t))-\mathcal{C}'(\u^*(t))(\z(t)-\u^*(t)),\p(t)\rangle\d t+(\z(T)-\u^*(T),\p(T)) =0.
	\end{align}
A calculation similar to \eqref{511} yields 
	\begin{align}\label{420a}
\langle\B(\z)- \B(\u^*)- \B'(\u^*)(\z-\u^*),\p\rangle=\langle\B(\z-\u^*,\z-\u^*),\p\rangle. 
	\end{align}
	Using Taylor's formula (see \eqref{416}), we obtain 
	\begin{align}
	&\beta\langle\mathcal{C}(\z)-\mathcal{C}(\u^*)-\mathcal{C}'(\u^*)(\z-\u^*),\p\rangle\nonumber\\&= \frac{\beta}{2}\left<\mathcal{C}''(\u^*)((\z-\u^*)\otimes(\z-\u^*)),\p\right>\nonumber\\&\quad+\frac{\beta}{2}\left<\int_0^1(1-\theta)^{2}(\mathcal{C}'''(\u^*+\theta(\z-\u^*))((\z-\u^*)\otimes(\z-\u^*)\otimes(\z-\u^*)))\d\theta,\p\right>\nonumber\\&=2\beta\langle(\u^*\cdot(\z-\u^*))(\z-\u^*),\p\rangle +\beta\langle|\z-\u^*|^2\u^*,\p\rangle+\beta\langle|\z-\u^*|^2(\z-\u^*),\p\rangle.
	\end{align}
	Note that $(\u^*,\U^*)$ satisfies the first order necessary condition 
	$$\int_0^T(\U(t),\D^*\p(t))_{\mathbb{U}}\d t+\int_0^T(\U(t),\U^*(t))_{\mathbb{U}}\d t=0, \ \text{ for all }\ \U\in\mathrm{L}^2(0,T;\mathbb{U}).$$ 
	 It is true for any $\U\in\mathscr{U}_{\mathrm{ad}}$, and in particular for $\W-\U^*\in\mathscr{U}_{\mathrm{ad}}$, that is, we have 
	\begin{align}\label{420}
	\int_0^T(\W(t)-\U^*(t),\D^*\p(t) )_{\mathbb{U}}\d t= 	-\int_0^T(\W(t)-\U^*(t),\U^*(t))_{\mathbb{U}}\d t.
	\end{align}
	Using the adjoint system \eqref{adjp} and \eqref{420a}-\eqref{420} in \eqref{419}, we further find 
	\begin{align}\label{5.35}
	& \int_0^T \langle\z(t)-\u^*(t),(\u^*(t)-\u_d(t))-\Delta\u^*(t)\rangle \d t + \int_0^T (\W(t) - \U^*(t),\U^*(t) )_{\mathbb{U}} \d t \nonumber\\&\quad+(\z(T)-\u^*(T),\u^*(T)-\u^f)  +\int_0^T(\langle\B(\z(t)-\u^*(t),\z(t)-\u^*(t)),\p(t)\rangle\d t\nonumber\\&\quad+\beta\int_0^T2\langle(\u^*(t)\cdot(\z(t)-\u^*(t)))(\z(t)-\u^*(t)),\p\rangle +\langle|\z(t)-\u^*(t)|^2\z(t),\p(t)\rangle\d t\nonumber\\&=0.
	\end{align} 
	An easy calculation yields 
	\begin{align}\label{3.25}
	&\frac{1}{2}\int_0^T \left[\|\z(t)-\u_d(t)\|_{\H}^2+\|\nabla\z(t)\|_{\H}^2+\|\W(t)\|_{\H}^2\right] \d t+\frac{1}{2}\|\z(T)-\u^f\|_{\H}^2 \nonumber\\&\quad- \frac{1}{2} \int_0^T \left[ \|\u^*(t)-\u_d(t)\|_{\H}^2+\|\nabla\u^*(t)\|_{\H}^2+\|\U^*(t)\|_{\H}^2\right] \d t-\frac{1}{2}\|\u^*(T)-\u^f\|_{\H}^2\nonumber \\	&= \int_0^T \left[(\u^*(t)-\u_d(t),\z(t)-\u^*(t) ) + \frac{1}{2}\|\z(t)-\u^*(t)\|_{\H}^2\right] \d t\nonumber\\&\quad+\int_0^T\left[\langle-\Delta\u^*(t),\z(t)-\u^*(t)\rangle+\frac{1}{2}\|\nabla(\z(t)-\u^*(t))\|_{\H}^2\right]\d t \nonumber\\	&\quad +\int_0^T \left[( \U^*(t),\W(t)-\U^*(t))_{\mathbb{U}}+ \frac{1}{2}\|\W(t)-\U^*(t)\|_{\mathbb{U}}^2\right] \d t\nonumber\\&\quad +(\u^*(T)-\u^f(T),\z(T)-\u^*(T) ) + \frac{1}{2}\|\z(T)-\u^*(T)\|_{\H}^2.
	\end{align}
Using \eqref{5.35} in \eqref{3.25}, we get 
	\begin{align*}
	&\frac{1}{2}\int_0^T \left[\|\z(t)-\u_d(t)\|_{\H}^2+\|\nabla\z(t)\|_{\H}^2+\|\W(t)\|_{\H}^2\right] \d t+\frac{1}{2}\|\z(T)-\u^f\|_{\H}^2 \nonumber\\&\quad- \frac{1}{2} \int_0^T \left[ \|\u^*(t)-\u_d(t)\|_{\H}^2+\|\nabla\u^*(t)\|_{\H}^2+\|\U^*(t)\|_{\H}^2\right] \d t-\frac{1}{2}\|\u^*(T)-\u^f\|_{\H}^2\nonumber \\	&=\frac{1}{2} \int_0^T \left[\|\z(t)-\u^*(t)\|_{\H}^2+\|\nabla(\z(t)-\u^*(t))\|_{\H}^2+\|\W(t)-\U^*(t)\|_{\mathbb{U}}^2\right] \d t\nonumber\\&\quad+ \frac{1}{2}\|\z(T)-\u^*(T)\|_{\H}^2-\int_0^T(\langle\B(\z(t)-\u^*(t),\z(t)-\u^*(t)),\p(t)\rangle\d t\nonumber\\&\quad+\beta\int_0^T2\langle(\u^*(t)\cdot(\z(t)-\u^*(t)))(\z(t)-\u^*(t)),\p\rangle +\langle|\z(t)-\u^*(t)|^2\z(t),\p(t)\rangle\d t\nonumber\\& \geq 0,
	\end{align*}
	where we used \eqref{417}. Therefore, for any $(\z,\W) \in \mathcal{A}_{\mathrm{ad}}$, we find that the following inequality holds:
	\begin{align*}
	&\frac{1}{2}\int_0^T \left[\|\z(t)-\u_d(t)\|_{\H}^2+\|\nabla\z(t)\|_{\H}^2+\|\W(t)\|_{\H}^2\right] \d t+\frac{1}{2}\|\z(T)-\u^f\|_{\H}^2 \nonumber\\&\geq  \frac{1}{2} \int_0^T \left[ \|\u^*(t)-\u_d(t)\|_{\H}^2+\|\nabla\u^*(t)\|_{\H}^2+\|\U^*(t)\|_{\H}^2\right] \d t+\frac{1}{2}\|\u^*(T)-\u^f\|_{\H}^2,
	\end{align*}
	which implies that the pair $(\u^*,\U^*)$ is an optimal pair for the problem \eqref{control problem}. 
\end{proof}
\begin{remark}
{	Note that the second order necessary and sufficient conditions of optimality (cf. \eqref{417}) is true for all $\alpha,\beta\geq 0$ and 2D CBF equations is the damped 2D NSE. Thus as $\alpha,\beta\to 0$, it coincides with the  second order necessary and sufficient conditions of optimality for the 2D NSE available in \cite{LijuanPezije}. }
\end{remark}

\section{Data Assimilation Problem}\label{sec6}\setcounter{equation}{0}  In this section, we consider a problem similar to that of the data assimilation problems of meteorology. In the data assimilation problems arising from meteorology, determining the accurate initial data for the future predictions is a challenging problem. This motivates us to consider a similar problem  for the 2D CBF equations.  We formulate an optimal data initialization problem, where we find the unknown optimal initial data by minimizing suitable cost functional  subject to the 2D CBF equations (see \cite{sritharan} for the case of the incompressible Navier-Stokes equations, \cite{BDM2} for the Cahn-Hillard-Navier-Stokes equations and \cite{MTM5} for the 2D Tidal dynamics equations).  

We formulate the initial data optimization problem as finding an optimal initial velocity  $\U\in\H$ such that $(\u, \U) $ satisfies the following system:
\begin{equation}\label{5p1}
\left\{
\begin{aligned}
\partial_t\u(t)+\mu\A\u(t)+\B(\u(t))+\alpha\u(t)+\beta\mathcal{C}(\u(t))&=\f(t),\ \text{ in }\ \V', \\
\u(0)&=\U,
\end{aligned}
\right.
\end{equation}
a.e. $t\in[0,T]$,  and minimizes the cost functional 
\begin{equation}\label{cost4}
\begin{aligned}
\mathscr{J}(\u,\U)&:=  \frac{1}{2} \|\U\|^2_{\H}+\frac{1}{2} \int_0^T \|\u(t) - \u_M(t)\|^2_{\H} \d t+\frac{1}{2}\int_0^T\|\nabla\times\u(t)\|_{\H}^2\d t +\frac{1}{2}\|\u(T)-\u_M^f\|^2_{\H},
\end{aligned}
\end{equation}
where $\u_M$ is the measured velocity of the fluid and $\u_M^f$ is the measured final velocity at time $T$. We assume that 
\begin{align}\label{um}\u_M\in\mathrm{L}^2(0,T;\H),\ \u_M^f\in\H.\end{align}
In this context, we take the set of admissible control class, $\mathscr{U}_{\mathrm{ad}}=\H$.  Also, the  \emph{admissible class} $\mathscr{A}_{\mathrm{ad}}$ consists of all pairs $(\u,\U)$ such that the set of states $\u(\cdot)$ is a unique weak solution of the 2D CBF equations \eqref{5p1} with the control $\U \in \mathscr{U}_{ad} $.  We formulate the optimal control problem as:
\begin{align} \label{IOCP}\tag{$\mathrm{P}_0$}
\min_{(\u,\U) \in \mathscr{A}_{\mathrm{ad}} } \mathscr{J}(\u,\U).
\end{align}

The next theorem gives the existence of  an optimal pair $(\u^*,\U^*)$ for the problem \eqref{IOCP}. A proof of the Theorem can be established in a similar way as that of the Theorem \ref{optimal}.
\begin{theorem}[Existence of an optimal pair]\label{optimal1}
	Let $\f\in\mathrm{L}^2(0,T;\V')$ be given. Then there exists at least one pair  $(\u^*,\U^*)\in\mathscr{A}_{\mathrm{ad}}$  such that the functional $ \mathscr{J}(\u,\U)$ given in \eqref{cost4} attains its minimum at $(\u^*,\U^*)$, where $\u^*(\cdot)$ is the unique weak solution of the system \eqref{5p1}  with the initial data control $\U^*\in\mathscr{U}_{\mathrm{ad}}$.
\end{theorem}
As  in the proof of the Theorem \ref{main}, the first order necessary conditions of optimality for this control problem also. Since $\mathscr{U}_{\mathrm{ad}}=\H$,  the optimal control is given by $\U^*=-\p(0)$, where $\p(\cdot)$ is the unique weak solution of the following adjoint system:  
	\begin{eqnarray}\label{adja}
\left\{
\begin{aligned}
-\partial_t\p(t)&+\mu\A\p(t)+(\B'(\u^*(t)))^*\p (t)+\alpha\p(t)+\beta\mathcal{C}'(\u^*(t))\p(t)\\&=(\u^*(t)-\u_M(t))-\Delta\u(t),\ \text{ in }\ \V', \\ \p(T)&=\u^*(T)-\u_M^f, 
\end{aligned}
\right.
\end{eqnarray} 
a.e. $t\in[0,T]$. A similar calculation as in the proof of Theorem  \ref{thm3.4} provides the existence of a unique weak solution to the system (\ref{adja}) such that  \begin{align}\label{348}\p\in \C([0,T];\H)\cap\mathrm{L}^2(0,T;\V).\end{align}
Using the continuity of $\p(\cdot)$ in time  at $t=0$ in $\H$, we infer that $\p(0)\in\H$ and hence we obtain  $\U^*=-\p(0)\in\mathscr{U}_{\mathrm{ad}}=\H.$
Thus, we have the following Theorem:

\begin{theorem}[Optimal initial control]\label{data}
	Let $(\u^*,\U^*)\in\mathscr{A}_{\mathrm{ad}}$ be an optimal pair. Then there exists a unique weak solution $\p(\cdot)$ of the adjoint system \eqref{adja} satisfying \eqref{348} such that the optimal control is obtained as \begin{align}\label{in}\U^*=-\p(0)\in\mathscr{U}_{\mathrm{ad}}=\H.\end{align}
\end{theorem}
For $r=3$, the second order necessary and sufficient conditions of optimality can be obtained in a similar way as in the Theorems \ref{necessary} and \ref{sufficient}.
\begin{theorem}[Necessary condition]\label{thm6.3}
	Let  $(\u^*,\U^*)$ be an optimal pair for the problem \eqref{IOCP}. Then, for any $(\u,\U) \in \mathcal{M}_{\u^*,\U^*}$, where $\mathcal{M}_{\u^*,\U^*}:= \{(\u,\U) \in \mathcal{A}_{\mathrm{ad}} \} - \{(\u^*,\U^*)\},$ we have 
	\begin{align}\label{6.7}
	&\frac{1}{2}\left[\int_0^T\left(\|\u(t)\|_{\H}^2+\|\nabla\u(t)\|_{\H}^2\right)\d t \right]+\frac{1}{2}\|\U\|_{\H}^2+\frac{1}{2}\|\u(T)\|_{\H}^2-\int_0^T\langle\B(\u(t),\u(t)),\p(t)\rangle\d t\nonumber \\ &\quad-\beta\int_0^T\left<2(\u^*(t)\cdot\u(t))\u(t)+|\u(t)|^2(\u(t)+\u^*(t)),\p(t)\right>\d t\geq 0. 
	\end{align}
\end{theorem}

\begin{theorem}[Sufficient condition]\label{thm6.4}
	Let  $(\u^*,\U^*)$ be a feasible pair for the problem \eqref{IOCP} and assume that the first order necessary condition holds (see  \eqref{in}). Let us also assume that for any $(\u,\U) \in \mathcal{M}_{\u^*,\U^*},$ the  inequality \eqref{6.7} holds. 	Then $(\u^*,\U^*)$ is an optimal pair for the problem \eqref{IOCP}.
\end{theorem}

 \medskip\noindent
{\bf Acknowledgements:} M. T. Mohan would  like to thank the Department of Science and Technology (DST), India for Innovation in Science Pursuit for Inspired Research (INSPIRE) Faculty Award (IFA17-MA110).  The author sincerely would like to thank the reviewer for his/her valuable comments and suggestions.

\end{document}